\documentclass[11pt]{amsart}
\usepackage[utf8]{inputenc}
\usepackage[ngerman,english]{babel}

\title[Complexity of Kazhdan--Lusztig varieties]{Complexity of the usual torus action on Kazhdan--Lusztig varieties}
\author{Maria Donten-Bury}
\address{Instytut Matematyki UW, Banacha 2, 02-097 Warszawa, Poland}
\email{M.Donten@mimuw.edu.pl}

\author{Laura Escobar}
\address{Department of Mathematics and Statistics\\ 
Washington University in St.\ Louis\\ One Brookings Drive \\ St.\ Louis, Missouri 63130 \\ U.S.A. }
\email{laurae@wustl.edu}

\author{Irem Portakal}
\address{Technische Universität München, Lehrstuhl für Mathematische Statistik 85748 Garching b. München, Boltzmannstr. 3.}
\email{mail@irem-portakal.de}

\date{\today}

\usepackage{fullpage}

\usepackage{graphicx}
\usepackage{amsmath}             
\usepackage{amsfonts}
\usepackage{float}

\usepackage[nomain,style=long3col,sort=use,section=section]{glossaries}

\usepackage{setspace}

\newlength\tindent
\usepackage{amsthm}               
\usepackage{amssymb}              
\usepackage{tikz}                   
\usepackage[all]{xy}
\usetikzlibrary{plotmarks}
\usetikzlibrary{shapes}
\usetikzlibrary{arrows}

\usepackage{upgreek}
\usepackage{ gensymb }
\usepackage{capt-of}
\usepackage{ stmaryrd }
\usepackage[pdftex]{hyperref}
\hypersetup{
    colorlinks=true,
    linkcolor={green!30!black},
    citecolor={blue!50!black},
    urlcolor={blue!80!black}
    allbordercolors={1 0 0}
}
\usepackage{xcolor}
\usepackage{caption}
\usepackage{young}
\usepackage{ytableau}
\usepackage{cleveref}
\usepackage{color}

\theoremstyle{plain}

\newtheorem{thm}{Theorem}[section]
\newtheorem{thm*}{Theorem}
\newtheorem{lem}[thm]{Lemma}

\newtheorem{prop}[thm]{Proposition}
\newtheorem{cor}[thm]{Corollary}

\theoremstyle{definition}

\newtheorem{rem}[thm]{Remark}
\newtheorem{ex}[thm]{Example}

\newtheorem{qu}[thm]{Question}
\newtheorem{defn}[thm]{Definition}

\DeclareMathOperator{\id}{id}
\DeclareMathOperator{\Spec}{Spec}
\DeclareMathOperator{\cone}{Cone}

\DeclareMathOperator{\SW}{SW}
\DeclareMathOperator{\TV}{TV}

\DeclareMathOperator{\dom}{dom}
\DeclareMathOperator{\Ess}{Ess}

\DeclareMathOperator{\KL}{\mathcal{N}}

\newcommand{\RR}{\mathbb{R}}      % for Real numbers
\newcommand{\ZZ}{\mathbb{Z}}      % for Integers
\newcommand{\CC}{\mathbb{C}}      % for Complex numbers
\newcommand{\oDia}{D\degree}

\newcommand{\rank}{r}

\begin{document}

\maketitle

\makeatletter
\def\blfootnote{\xdef\@thefnmark{}\@footnotetext}
\makeatother

\blfootnote{\emph{MSC2020 Subject Classification:} Primary: 14M15, 14M25, 52B20, 05E10; Secondary: 05C20.}


\begin{abstract}
We investigate the class of Kazhdan--Lusztig varieties, and its subclass of matrix Schubert varieties, endowed with a naturally defined torus action. 
Writing a matrix Schubert variety $\overline{X_w}$ as $\overline{X_w}=Y_w\times \CC^d$ (where $d$ is maximal possible), 
we show that $Y_w$ can be of complexity-$k$ exactly when $k\neq 1$.
Also, we give a combinatorial description of the extremal rays of the weight cone of a Kazhdan--Lusztig variety, which in particular turns out to be the edge cone of an acyclic directed graph. As a consequence we show that given permutations $v$ and $w$, the complexity of Kazhdan--Lusztig variety indexed by $(v,w)$ is the same as the complexity of the Richardson variety indexed by $(v,w)$. Finally, we use this description to compute the complexity of certain Kazhdan--Lusztig varieties.
\end{abstract}

\setcounter{tocdepth}{1}
\tableofcontents

\section{Introduction}

A group action on a variety very often gives significant information about its structure and allows one to simplify its description. 
A very well-known example of this phenomenon is the class of toric varieties, i.e.\ algebraic varieties endowed with an action of an algebraic torus $(\CC^*)^n$ such that there is a dense orbit. A toric variety can be completely described using combinatorial objects: a fan of convex polyhedral cones in $\mathbb{R}^n$, rational with respect to a lattice $\mathbb{Z}^n \subset \mathbb{R}^n$.
In~\cite{p-div_AH} Altmann and Hausen extended this description to the class of $T$-varieties, that is algebraic varieties with an algebraic torus action which does not necessarily have a dense orbit. In general, a $T$-variety $X$ can be presented in terms of its \emph{p-divisor} which is a partially combinatorial object, consisting of a special quotient of $X$ by the considered action (the geometric part) and a fan of rational polyhedral cones in the lattice related to the torus (the combinatorial part). If the complexity of the torus action, i.e.\ the codimension of the biggest orbit, is small, then the combinatorial part contains a lot of information about $X$. In particular, if complexity is~0 then $X$ is toric and the p-divisor is purely combinatorial. 

While toric varieties and their combinatorial description have been widely investigated, there is still a lot to find out about $T$-varieties of higher complexity, even in the case of complexity-1. 
Their geometric structure in general has been studied in~\cite{geometryofT}, but, in particular, there are not many non-trivial examples of p-divisors. 
Apart from the ones in~\cite{p-div_AH}, one may look at p-divisors coming from the spectrum of the Cox ring structure on a variety, see~\cite{p-div_AW}. 
An interesting general question is to find $T$-varieties such that their p-divisors are not too difficult to determine, but they do not come from a structure of a toric variety by downgrading the action. In particular, one may look for a complexity-1 $T$-variety such that the quotient is a curve of higher genus. Such examples could be later used to investigate the relation between the p-divisor description of a $T$-variety and its degenerations, see~\cite{p-div_AW}. Thus, p-divisors give a motivation for studying $T$-varieties of small complexity.

A related problem is to describe a given collection of varieties with a torus action in the language of $T$-varieties of \cite{p-div_AH}.
The goal of this paper is to begin to carry out this program for a class of algebraic varieties, called \emph{Kazhdan--Lusztig varieties}, defined by imposing certain determinantal conditions to matrices in $\CC^{n\times n}$, see Section~\ref{def_KL}.
Such a variety, denoted $\KL_{v,w}$ depends on two permutations $v,w \in S_n$, where $S_n$ denotes the symmetric group of degree $n$. 
This class of varieties has been used to describe local behavior of Schubert varieties in the flag variety, e.g.~in \cite{WY08}, and to
provide geometric explanations for various combinatorial phenomena in algebraic combinatorics, see e.g.~\cite{factschu,wooyong,CCLMZ,KW}.

Kazhdan--Lusztig varieties are endowed with a torus action, which we call the \emph{usual torus action} and describe in Section~\ref{sec: usualtorusaction}. We investigate the complexity of this action and present methods for finding classes of Kazhdan--Lusztig varieties of small complexity using results from representation theory and cones arising from directed graphs.
A related problem is to investigate the complexity of Schubert varieties, e.g.~\cite{Karuppuchamy,leemasudapark}.
As we will see in Sections~\ref{sec: klvarieties} and \ref{lowcomplexityexamples}, only in special cases does the complexity of a Schubert variety $X_w$ determine the complexity of Kazhdan--Lusztig varieties $\KL_{v,w}$.

We now describe the structure of the paper as well as the main results.
In Section~\ref{torusactiononmsv}, we briefly introduce $T$-varieties and toric ideals arising from directed graphs. We present a formula for the dimension of the \emph{edge cone} associated to this toric variety, which is used to determine the complexity of a $T$-variety. In Section~\ref{sec_torus_action_MSch} we study a special subclass of Kazhdan--Lusztig varieties: \emph{matrix Schubert varieties} which were introduced by Fulton in \cite{fulton}.
A matrix Schubert variety $\overline{X_w}$ is isomorphic to $Y_w \times \CC^d$ where $d$ is as large as possible and $Y_w$ is an affine variety. We show in Section~\ref{sec: complexity of the torus action on MSV} that the weight cone of the usual torus action on $Y_w$ is the edge cone of an acyclic bipartite graph. In~\cite{escobarmeszaros} the authors give a combinatorial criterion for $Y_w$ to be toric. By simple arguments using graphs to investigate the dimension of the weight cone, we hoped to classify complexity-1 matrix Schubert varieties, but surprisingly it turns out that there are none, while all other complexities can be obtained.
\begin{thm*}[Theorem~\ref{nocomplexity1matrixschu}, Theorem~\ref{othercomplexitymsv}]
There exist $T$-varieties $Y_w$ of complexity-$k$ with respect to usual torus action only for $k \neq 1$, where $\overline{X_w} \cong Y_w \times \mathbb{C}^d$ is a matrix Schubert variety with $d$ maximal possible.
\end{thm*}
In Section~\ref{sec: klvarieties}, we investigate properties of the usual torus action on general Kazhdan--Lusztig varieties. The main result of this section is the description of the generators of the weight cone of the usual torus action on~$\KL_{v,w}$. The description is combinatorial, it relies on the properties of permutation $v,w$ which determine~$\KL_{v,w}$. Given a permutation $v$ denote by $\ell(v)$ the length of the permutation, i.e.\ the number of inversions of $w$.

\begin{thm*}[Theorem~\ref{permittingdeletingedges}]
The weight cone of the usual torus action on $\KL_{v,w} \cong \overline{X_w}\cap \Omega_v^{\circ}$ is spanned by weights $e_{v(j)} - e_{i}$ corresponding to the coordinate $z_{ij}$ of the opposite Schubert cell $\Omega_v^\circ$ such that $t_{v(j),i}v \leq w$ and $\ell(t_{v(j),i}v) - \ell(v) = 1$, where $t_{v(j),i}$ is the permutation transposing $v(j)$ and~$i$.
\end{thm*}

One observes immediately that this weight cone is also an edge cone of an acyclic graph. This result, joined with the methods of investigating torus action via associated graphs, gives a tool which significantly simplifies the task of finding the complexity of a Kazhdan--Lusztig variety with the usual torus action. 
As a consequence of this theorem, we show in \Cref{cor_richi_KL} that the complexity of $\KL_{v,w}$ is the same as the complexity of the Richardson variety indexed by $(v,w)$.

Lastly, Section~\ref{lowcomplexityexamples} is a study of special cases with interesting examples of $T$-varieties.  
More precisely, given $d\in\mathbb{N}$ we give a classification for the Kazhdan--Lusztig varieties of complexity-$d$ when $v$ is the identity or a simple reflection.
Concretely, in Theorem~\ref{visid} and Theorem~\ref{vistransp} we give a formula for the complexity in terms of a reduced word expression of~$w$. 

\section{Toric Background}\label{torusactiononmsv}

\subsection{$T$-varieties of complexity-$d$}
\label{sec-general-t-varieties}

Let us discuss some relevant background on $T$-varieties.
For a deeper discussion of $T$-varieties, we refer the reader to \cite{p-div_AH} and \cite{geometryofT}.

\noindent Given a torus $T$, let ${\sf M}(T)$ denote the character lattice of $T$ and ${\sf M}(T)_{\RR}$ the real vector space spanned by ${\sf M}(T)$.
\begin{defn}
An affine normal variety $X$ is a \textit{T-variety} of complexity-$d$ if it admits an effective $T$ torus action with $\text{dim}(X) - \text{dim}(T) = d$.
\end{defn}
Note that normal affine toric varieties are $T$-varieties of complexity zero.

Let us describe how to compute the complexity of a $T$-variety $X$ using polyhedral methods.
The convex polyhedral cone generated by all weights of the torus action on $X$ in ${\sf M}(T)_{\RR}$ is called the \emph{weight cone} of the action.
Let $p$ be a general point in $X$.
The closure of the torus orbit $\overline{T\cdot p}$ is the affine normal toric variety associated to the weight cone of the action.
Therefore, the dimension of $\overline{T\cdot p}$ is equal to the dimension of the weight cone of the action.
Since the action of $T$ on $\overline{T\cdot p}$ is effective, $\dim(T)=\dim(\overline{T\cdot p})$ and it follows that the dimension of the weight cone is equal to the complexity of the $T$-action.
We conclude that the complexity of the $T$-variety $X$ is equal to $\dim(X)$ minus the dimension of the weight cone. If the $T$-action is not effective then the subgroup $S=\bigcap_{p\in X}T_p$ is nontrivial, where $T_p=\{t\in T \mid t\cdot p=p\}$. 
The induced $T/S$-action on $X$ is given by $tS\cdot p=t\cdot p$ and it is both well defined and effective.
Moreover, the weights of the $T/S$-action are the same as the weights of the $T$-action. 
It follows that the complexity of the $T/S$-action is also equal to $\dim(X)$ minus the dimension of the weight cone. 

Throughout this paper, whenever we have a $T$-action on $X$ that is not effective, we will abuse notation and refer to $X$ as a $T$-variety with complexity equal to that of the $T/S$-action.

\subsection{Toric ideals arising from directed graphs}\label{sec-toric-graph}
Since we work with directed acyclic graphs throughout this paper, we restrict our attention to these graphs, despite the fact that the following construction can be done for general directed graphs. \\
Let $G$ be a connected directed acyclic graph with $n$ vertices and $q$ edges. 
We denote the edge set by $E(G)$ and its vertex set by $V(G)$. We choose the notation $(a \to b)$ for the directed edge from vertex $a$ to vertex $b$ in order to avoid a conflict with other notation. 
The kernel $I_G$ of the following morphism
\begin{align*}
\Phi_G \colon \CC[x_1, \ldots, x_{q} ]&\rightarrow \CC[t_1^{\pm{1}}\ldots, t_{n}^{\pm{1}}] \\
x_{e} &\mapsto t_i t_j^{-1},
\end{align*}
where $e=(i \to j) \in E(G)$, is called the \emph{edge ideal} of $G$, see \cite{hibi-ohsugi99,gitler}.
The edge ideal is prime and generated by binomials.
Concretely, let $c :=(v_{i_1}, \ldots, v_{i_{r}})$ be a (not directed) cycle i.e.\ a closed walk in which only the first and the last vertices are equal. Here $v_{i_{k}}$ is an edge of $G$, for $k \in [r]$. We split the cycle into two disjoint sets of edges $c_{+}$ and $c_{-}$ where $c_{+}$ consists of the clockwise oriented edges in the cycle $c$ and $c_{-}=c \setminus c_{+}$. Then we define the following binomial $$f_{c}:= \prod_{v_i \in c_{+}} x_{i} - \prod_{v_i \in c_{-}} x_{i}\in \CC[x_1, \ldots, x_q].$$ If $c_{+}=\emptyset$ or $c_{-}=\emptyset$, we set $\prod_{v_i \in c_{+}} x_{i}=1$ or $\prod_{v_i \in c_{-}} x_{i}=1$. We say that a cycle $\Gamma$ is primitive if it does not have any chord i.e.\ an edge which is not part of the cycle but connects two vertices of the cycle.
\begin{thm}\cite[Cor 4.5]{gitler}\label{thm-primitive}
The edge ideal $I_G$ associated to $G$ is generated by the binomials $f_{c}$ where $c$ is a primitive cycle.
\end{thm}

The affine normal toric variety associated to $G$ is defined as 
    $$\TV(G):=\Spec(\CC[x_1, \ldots, x_{q}]\big / I_G)=\Spec(\CC[\sigma^{\vee}_G \cap {\sf M}(T) ]),$$ 
where $\sigma^{\vee}_G \subseteq \sf M(T)_{\RR}$ is called the (dual) \emph{edge cone}.
Let $e_i \in \mathbb{R}^{|V(G)|}$ be the $i$th standard basis vector. One can describe the dual edge cone precisely as $$\sigma^{\vee}_G=\cone(e_i - e_j | \ (i \to j) \in E(G)).$$

\noindent The following result is crucial while calculating the complexity of a $T$-variety throughout this paper.
\begin{lem}\label{dimofcone}
Let $G$ be a directed acyclic graph with $n$ vertices and $k$ connected components. The dimension of the edge cone $\sigma_G^{\vee} \subseteq \sf M_{\RR}$ is $n-k$.
\end{lem}
\begin{proof}
The rays of $\sigma_G^\vee$ are the columns of the incidence matrix $A_G$ of $G$, so the dimension of $\sigma_G^\vee$ equals the rank of $A_G$.
By \cite[Lem 4.1]{gitler} if $G$ is connected, then the rank of the incidence matrix $A_G$ is $n-1$. For $G$ with connected components $G_1,\ldots,G_k$ after a relabelling of the vertices and edges we obtain 
\[A_G=\begin{bmatrix}
    A_{G_1}       &0 & 0 & \dots & 0 \\
    0           & A_{G_2}& 0 & \dots & 0 \\
   \vdots& 0& \ddots & & 0\\
   \vdots&\vdots&& \ddots& \vdots \\
    0      & 0 & 0& \dots & A_{G_k}
\end{bmatrix}.\]
Hence the rank of $A_G$ is $n-k$.
\end{proof}

%%%%%%%%%%%%%%%%%%%%%%%%%%%%%%%%%%%%%%%%
\section{Torus action on matrix Schubert varieties}\label{sec_torus_action_MSch}

\noindent Throughout this section we choose conventions which are compatible with \cite{leemasuda} and \cite{wooyong}. 

\subsection{Background on matrix Schubert varieties}
Let $w \in S_n$ be a permutation. We denote its permutation matrix again by $w \in \mathbb{C}^{n\times n}$ with the convention 
 \begin{equation}
w_{ij} :=
    \begin{cases}
     1,& \text{if } w(j)=i, \\
      0, & \text{otherwise.}
    \end{cases}
  \end{equation}

Consider the following action of $B \times B$ on $\mathbb{C}^{n\times n}$, where $B$ is the group of invertible upper triangular matrices of size $n\times n$: 
\begin{align*}
\begin{split}
(B \times B) \times \mathbb{C}^{n\times n}&\rightarrow \mathbb{C}^{n\times n} 
\\
((X, Y), M) &\mapsto XMY^{-1}
\end{split}
\end{align*}
Given $a,b\in[n]:=\{1,2,\ldots,n\}$ and $M\in\mathbb{C}^{n\times n}$ let $M_{\square}^{a,b} \in \mathbb{C}^{(n-a+1) \times b}$ be the submatrix of $M$ consisting of rows $a,\ldots,n$ and columns $1,\ldots,b$. 
We denote the rank of $M_{\square}^{a,b}$ as $\rank_M(a,b)$.
We remark that $M \in B w B$ if and only if $\rank_M(a,b) = \rank_w(a,b)$ for all $a,b \in [n]$.  
\begin{defn}\label{def_mSv}
The matrix Schubert variety associated to permutation $w \in S_n$ is the Zariski closure $\overline{X_w}:=\overline{B w B}\subset\mathbb{C}^{n\times n}$.
\end{defn}

\noindent These varieties were first investigated in \cite{fulton}, where their defining ideals were described combinatorially in terms of Rothe diagrams.

\begin{defn}\label{rothediagram}
The \emph{Rothe diagram} of $w \in S_n$ is $D(w):= \{(w(j), i) \ | \ i <j, \ w(i)>w(j)\}$. 
\end{defn}
Note that each entry in $D(w)$ is in one-to-one correspondence with an inversion of $w$ and therefore $\ell(w) = |D(w)|$, where $\ell(w)$ is the Coxeter length of $w$.
Starting with the permutation matrix $w$, $D(w)$ is visualized by replacing each $1$ by a $\bullet$, deleting all $0$s, and placing a box in each entry of the permutation matrix $w$ that is not south or east of a $\bullet$.

\begin{defn}\label{def_opposite_Rothe}
The \emph{opposite Rothe diagram} of $w \in S_n$ is $D\degree(w):=\{(i,j) \ | \ w(j)<i, \ w^{-1}(i)>j\}$.
\end{defn}
\noindent In this case, $D\degree(w)$ is visualized in a similar way to $D(w)$, except that we place a box on each entry that is not north or east of a $\bullet$.
Note that $D\degree(w)=\{(w(i),j) \ | \ j<i, \ w(j)<w(i)\}$, 
    $$\ell(w) = |D(w)| = \frac{n(n-1)}{2} - |D\degree(w)|,$$
and the connected components of $D\degree(w)$ are Young diagrams in French notation.

\begin{defn} The set consisting of the north-east corners of each connected component of $D \degree(w)$ is called the \emph{essential set} of $w$ and denoted $\Ess(w)$. 
\end{defn}

\begin{thm}[{\cite[Prop.~3.3, Lem.~3.10]{fulton}}] \label{fultonmatrixschubert}
The matrix Schubert variety $\overline{X_w}$ is an affine variety of dimension $n^2 - |D\degree(w)|$. It is defined as a scheme by the determinants encoding the inequalities $r_M(a,b) \leq r_w(a,b)$ for all $(a,b) \in \Ess(w)$.
\end{thm}

\noindent Matrix Schubert varieties are normal varieties, see e.g.\ \cite[Thm.~2.4.3]{factschu}.
\begin{ex}\label{basicexampletomaybedeleted}
Let $w=3412 \in S_4$. We obtain the traditional and opposite Rothe diagrams as in Figure~\ref{rotheoppositefigure}. The length of $w$ is $|D(w)| = 4 = 6 - |D\degree(w)|$.
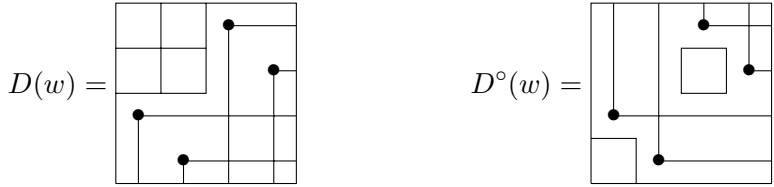
\begin{figure}[h]
	$$
	D(w) = 
	\begin{tikzpicture}[baseline=(O.base),scale=1.2]
	\node(O) at (1,1) {};
	\draw (0,0)--(2,0)--(2,2)--(0,2)--(0,0)
		(.5,1)--(.5,2)
		(0,1.5)--(1,1.5)
		(0,1)--(1,1)--(1,2); 
	\draw 
		(.25,0)--(.25,.75) node {$\bullet$}--(2,.75)
		(.75,0)--(.75,.25) node {$\bullet$}--(2,.25)
		(1.25,0)--(1.25,1.75) node {$\bullet$}--(2,1.75)
		(1.75,0)--(1.75,1.25) node {$\bullet$}--(2,1.25);
	\end{tikzpicture}
	\qquad\qquad\qquad
	D\degree(w) = 
	\begin{tikzpicture}[baseline=(O.base),scale=1.2]
	\node(O) at (1,1) {};
	\draw (0,0)--(2,0)--(2,2)--(0,2)--(0,0)
		(0,.5)--(.5,.5)--(.5,0)
		(1.5,1)--(1,1)--(1,1.5)--(1.5,1.5)--(1.5,1); 
	\draw 
		(.25,2)--(.25,.75) node {$\bullet$}--(2,.75)
		(.75,2)--(.75,.25) node {$\bullet$}--(2,.25)
		(1.25,2)--(1.25,1.75) node {$\bullet$}--(2,1.75)
		(1.75,2)--(1.75,1.25) node {$\bullet$}--(2,1.25);
	\end{tikzpicture}
	$$
\caption{The Rothe and opposite Rothe diagram of $3412$.}\label{rotheoppositefigure}
\end{figure}
Observe that $\Ess(w) = \{(4,1),(2,3)\}$ and the matrix Schubert variety $\overline{X_w}$ is defined by the inequalities $r_M(4,1) \leq r_w(4,1) = 0$ and $r_M(2,3) \leq r_w(2,3) = 2$. Its defining ideal is then $\left( z_{41},\det (M_{\square}^{2,3}) \right) = \left( z_{41},\det\begin{pmatrix}
z_{21} & z_{22} & z_{23} \\
z_{31} & z_{32} & z_{33} \\
z_{41} & z_{42} & z_{43}
\end{pmatrix} \right) \subset \CC[z_{12},\ldots,z_{44}]$.
Note that $\dim(\overline{X_{w}}) = 14$.
\end{ex}

\noindent Given $w\in S_n$ there exists an affine variety $Y_w$ such that $\overline{ X_w} =Y_w\times \CC^d$ where $d$ is as large as possible. 
Let us describe $Y_w$ and $d$ more precisely, following \cite{escobarmeszaros} with the change of conventions we explain in Remark~\ref{rem:conventions}.
\begin{defn}
The connected component of $(n,1)$ in $D \degree(w)$ is the \emph{dominant piece} of $w$ and is denoted $\dom(w)$.
Denoted by $\SW(w)$ the set consisting of the $(i,j)$ which are south-west of some entry in $\Ess(w)$. Moreover we let $L(w):=\SW(w) \setminus \dom(w)$ and $L'(w):=L(w)\setminus D\degree(w)$.
\end{defn}

We remark that since the connected components of $D\degree(w)$ are Young diagrams, then $\dom(w)$ and $\SW(w)$ also are Young diagrams. 
It follows that $L(w)$ is a skew diagram.
However, as one can observe in $L'(5	1	4	2	3)$, $L'(w)$ may not be a skew diagram.

Note that $(a,b) \in \dom(w)$ if and only if $r_w(a,b)=0$. Also, by Theorem~\ref{fultonmatrixschubert}, the ideal defining the matrix Schubert variety depends only on the submatrices which are in $\SW(w)$.  
First, consider the projection of the matrix Schubert variety $\overline{X_{w}} \subseteq \CC^{n\times n}$ onto the entries which are not is $\SW(w)$. Since these entries are free in $\overline{X_w}$, the projection isomorphic to $\mathbb{C}^{n^2-|\SW(w)|}$. 
Define $Y_{w}$ as the projection onto the entries of $L(w)$.  
We obtain that $\overline{X_{w}} = Y_{w} \times \mathbb{C}^{n^2 - |\SW(w)|}$ and 
    \begin{equation}\label{eq-dim-yw}
    \dim (Y_{w}) = n^2 - |D \degree(w)| - (n^2 - |\SW(w)|) =|\SW(w)| - |D\degree(w)| = |L'(w)|.
    \end{equation}

\begin{ex}\label{swll}
Continuing with $w=3412 \in S_4$ from Example~\ref{basicexampletomaybedeleted}, $\dom(w)=\{(4,1)\}$.
Figure~\ref{youngtableauxfig} shows $\SW(w)$, $L(w)$, and $L'(w)$.
\begin{figure}
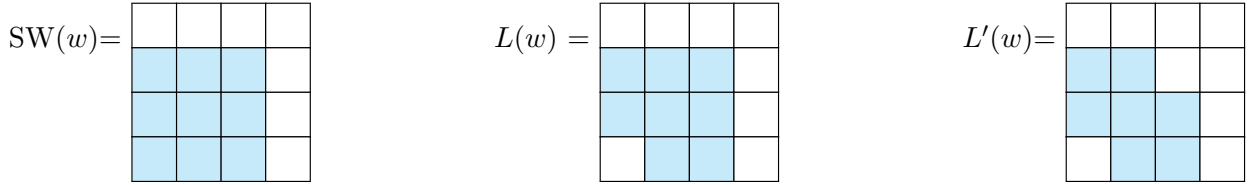

    \centering
$\SW(w)$=
\begin{ytableau}
*(white) & *(white) & *(white) &*(white)  \\
*(cyan!20)    & *(cyan!20)     &*(cyan!20)   &*(white)   \\
*(cyan!20)    & *(cyan!20)    & *(cyan!20)     &  *(white) \\
*(cyan!20)  & *(cyan!20)    & *(cyan!20)    &  *(white) 
\end{ytableau}
\qquad\qquad\qquad
$L(w)$ = \begin{ytableau}
*(white) & *(white) & *(white) &*(white)  \\
*(cyan!20)    & *(cyan!20)     &*(cyan!20)   &*(white)   \\
*(cyan!20)    & *(cyan!20)    & *(cyan!20)     &  *(white) \\
*(white)  & *(cyan!20)    & *(cyan!20)    &  *(white) 
\end{ytableau}
\qquad\qquad\qquad
$L'(w)$=
\begin{ytableau}
*(white) & *(white) & *(white) &*(white)  \\
*(cyan!20)    & *(cyan!20)     &*(white)   &*(white)   \\
*(cyan!20)    & *(cyan!20)    & *(cyan!20)     &  *(white) \\
*(white)  & *(cyan!20)    & *(cyan!20)  &  *(white) 
\end{ytableau}
\caption{$\SW(w)$, $L(w)$ and $L'(w)$ for $w=3412$.}\label{youngtableauxfig}
\end{figure}

We observe that $\overline{X_w}=Y_w\times \CC^7$, where $Y_w$ is the hypersurface defined by the ideal $$\left( \det\begin{pmatrix}
z_{21} & z_{22} & z_{23} \\
z_{31} & z_{32} & z_{33} \\
0 & z_{42} & z_{43}
\end{pmatrix} \right ) \subset \CC[z_{21},z_{22},z_{23},z_{31},z_{32},z_{33},z_{42},z_{43}].$$
 In particular, $\dim(Y_w) = 7$.
\end{ex}

\begin{rem}\label{rem:conventions}
The matrix Schubert varieties appearing in \cite{escobarmeszaros,portakal2} are defined using a $(B_{-} \times B)$-action on $\CC^{n\times n}$ where $B_{-}$ is the group of lower triangular matrices of size $n\times n$.
This changes the conventions, e.g.\ one imposes north-west rank conditions instead of south-west rank conditions.
We can translate from the conventions of this paper into theirs by reflecting the matrices horizontally, i.e.\ $w\mapsto w_0w$ where $w_0$ is the longest permutation. 
\end{rem}

%%%%%%%%%%%%%%%%%%%%%%%%%%%
\subsection{The complexity of the torus action on $Y_w$}\label{sec: complexity of the torus action on MSV}
We start by describing the torus action on $Y_w$ considered in this paper.
Note that $Y_w$ is isomorphic to the subvariety of $\overline{X_{w}}$ obtained by setting $z_{ij}=0$ for all $(i,j)\notin \SW(w)$. Clearly, the $B\times B$-action on $\overline{X_{w}}$ fixes this subvariety and, thus, induces a $B\times B$-action on $Y_w$.
In this section, we investigate the restriction to $T\times T$ of the action of $B \times B$ on $Y_{w}$, called the \emph{usual torus action}.

The $Y_w$ that are toric with respect to the usual torus action are characterized in \cite[Thm.~3.4]{escobarmeszaros}. 
In this section we further study the complexity of this action. Our main results are Theorem~\ref{nocomplexity1matrixschu}, which states that there exists no complexity-1 $T\times T$-variety $Y_w$, and Theorem~\ref{othercomplexitymsv}, which states that there exist $Y_w$ of complexity-$d$ for $d\ge2$. We use the notation ``$T\times T$-variety" in a non-traditional way to avoid any confusion over which torus action we consider on the variety $Y_w$. 

\noindent We begin by computing the weight cone of the $T\times T$-action on $Y_w$.
Note that $T\times T$ has character lattice ${\sf M}(T\times T)\cong \ZZ^n\times\ZZ^n$.
The weights of the $T\times T$-action on $\CC^{n\times n}=\Spec\CC[x_{11},\ldots,x_{nn}]$ are the $(m,\bar m)\in \ZZ^n\times\ZZ^n$ for which there exists $x_{ij}$ such that $(t,\bar t)\cdot x_{ij}=t^m\bar t^{\bar m}x_{ij}$ for all $(t,\bar t)\in T\times T$. 
Since $(t,\bar t)\cdot x_{ij}=t_i\bar t_j^{-1}x_{ij}$, the weights are $\{e_i-f_j\mid i,j\in[n]\}$, where $e_1,\ldots,e_n,f_1,\ldots,f_n$ denote the standard basis for $\ZZ^n \times \ZZ^n$. 
We conclude that the weight cone of the $T\times T$-action on $Y_w$ is
    $$\sigma_w = \cone(e_i - f_j \ | \ (i,j)\in L(w)).$$

\noindent In fact, this cone is an edge cone from \Cref{sec-toric-graph} and its generators are encoded by the edges of an acylic directed bipartite graph. The dimension of the weight cone can be calculated very easily via graphs (see \Cref{dimofcone}) and many combinatorial aspects of the edge cone are well-understood \cite{PORTAKAL2021784,Setiabrata2020FacesOR}. Hence, next we reformulate the considered torus action on~$Y_w$ in terms of graphs.

Note that we consider the $T\times T$-varieties with respect to the usual torus action defined as before. 
To ease notation, given $[n]\amalg[n]$ we write $\bar k$ for the elements of the second $[n]$.
Given $w\in S_n$, let $G^w$ be the acyclic bipartite graph with $E(G^w)=\{ (a \to \bar b) \mid (a,b)\in L(w)\}$ and  $V(G^w)\subseteq[n]\amalg [n]$ so that $G^w$ has no isolated vertices. Note that this assumption is not a strong one. Indeed, this does not change the edge cone, but only its ambient dimension.
By definition, we have that the weight cone of $Y_w$ is the dual edge cone of $G^w$, i.e.\ $\sigma_w=\sigma_{G^w}^{\vee}$. For simplicity, we refer to~$\sigma_w$ as the edge cone of $Y_w$.
It follows from our discussion that $Y_{w}$ is a $T\times T$-variety of complexity-$d$ with respect to the torus action $T\times T$ if and only if 
\begin{equation}\label{eq_complexity}
    \dim( \sigma_{w})
    =\dim(Y_w)-d
    =|L'(w)| -d,
\end{equation} 
where the second equality follows from \eqref{eq-dim-yw}.

\begin{rem}
The edge cone of the acyclic bipartite graph $G^w$ is isomorphic to the edge cone of the underlying undirected graph of $G^w$. This undirected graph is in particular used to classify rigid toric matrix Schubert varieties in \cite{portakal2}.
\end{rem}

\begin{ex}
Let us consider again the matrix Schubert variety $\overline{ X}_{3412} \cong Y_{3412} \times \CC^7$ from Example~\ref{swll}. For each box $(a,b)$ in $L(w)$, we construct an edge $ (a \to \bar b)$. The dimension of the associated edge cone $\sigma_{w}$ is 5 and $|L'(w)| = 7$. Hence $Y_{3412}$ is a $T\times T$-variety of complexity-2 with respect to the usual torus action.

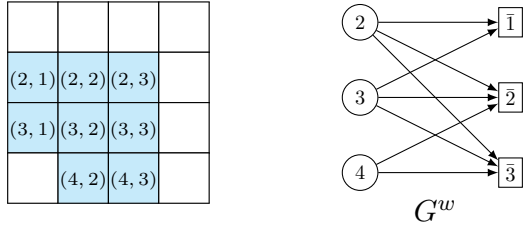
\begin{center}
\ytableausetup{smalltableaux, textmode, boxsize=1.7em}

\begin{ytableau}
*(white) & *(white) & *(white) &*(white)  \\
*(cyan!20) \tiny $(2,1)$  & *(cyan!20) \tiny $(2,2)$  &*(cyan!20)\tiny $(2,3)$ &*(white)   \\
*(cyan!20) \tiny $(3,1)$  & *(cyan!20)\tiny $(3,2)$  & *(cyan!20) \tiny $(3,3)$  &  *(white) \\
*(white)   & *(cyan!20) \tiny $(4,2)$ & *(cyan!20) \tiny $(4,3)$  &  *(white) 
\end{ytableau}
\hspace{1.5cm}
\begin{tikzpicture}[baseline=60,scale=1,every path/.style={>=latex},every node/.style={draw,circle,fill=white,scale=0.7}]
  \node            (a2) at (0,1.5)  {3};
  \node            (a3) at (0,0.5)  {4};
\node   [draw=none,fill=none,scale=1.5]         (label) at (1,0)  {$G^w$};

\node       [rectangle]     (b7) at (2,0.5)  {$\bar{3}$}; 
 \node       [rectangle]     (b6) at (2,1.5)  {$\bar{2}$};
  \node            (a1) at (0,2.5) {2};
  \node       [rectangle]     (b5) at (2,2.5) {$\bar{1}$};

    \draw[->] (a1) edge (b7);
  \draw[->] (a1) edge (b6);
  \draw[->] (a2) edge (b5);
  \draw[->] (a2) edge (b6);
    \draw[->] (a2) edge (b7);
    \draw[->] (a3) edge (b7);
    \draw[->] (a3) edge (b6);
    \draw[->] (a1) edge (b5);

\end{tikzpicture}

\end{center}
\end{ex}

\noindent The $Y_w$ of complexity-$0$, i.e.\ toric varieties, have been classified in \cite{escobarmeszaros}.
A \emph{hook} with corner $(i,j)$ consists of boxes $(i',j')$ such that $j=j'$ and $i' \geq i$ or $i=i'$ and $j \geq j'$.

\begin{thm}[{\cite[Thm 3.4]{escobarmeszaros}}]\label{toricschu}
$Y_{w}$ is a toric variety with respect to the $T\times T$-action if and only if $L'(w)$ consists of disjoint hooks not sharing a row or a column.
\end{thm}

\noindent Some simple arguments used in the alternative proof of this result in \cite[Thm 3.2]{portakal2} motivated us to work further with graphs to investigate the varieties $Y_w$ of larger complexity.

\begin{ex}
Let $w = 3142 \in S_4$. The \Cref{fig_toric_ex} illustrates $\oDia(w)$, $L(w)$, $L'(w)$, and $G^w$. The dimension of the associated bipartite graph and $|L'(w)|$ is three. Also, as seen in this figure, $L'(w)$ has a hook shape. Thus, $Y_w$ is a toric variety with respect to the effective torus action of $T = (\CC^*)^3$.

\begin{figure}[h]
\ytableausetup{smalltableaux, mathmode, boxsize=1.5em}
\begin{ytableau}
*(white)  & *(white)  &*(white)  &*(white)  \\
*(white)  & *(cyan!20)  &*(white) &*(white)   \\
*(white)  & *(white)  & *(white)   &  *(white) \\
*(cyan!20)   & *(cyan!20)   & *(white)   &  *(white) 
\end{ytableau}
\hspace{0.5cm}
\begin{ytableau}
*(white)  & *(white)  &*(white)  &*(white)  \\
*(cyan!20)  & *(cyan!20)  &*(white) &*(white)   \\
*(cyan!20)  & *(cyan!20)  & *(white)   &  *(white) \\
*(white)  & *(white)  & *(white)   &  *(white) 
\end{ytableau}
\hspace{0.5cm}
\begin{ytableau}
*(white)  & *(white)  &*(white)  &*(white)  \\
*(cyan!20)  & *(white)  &*(white) &*(white)   \\
*(cyan!20)  & *(cyan!20)  & *(white)   &  *(white) \\
*(white)   & *(white)  & *(white)   &  *(white) 
\end{ytableau}
\hspace{0.5cm}
\begin{tikzpicture}[baseline=2.1cm,scale=1.2,every path/.style={>=latex},every node/.style={draw,circle,fill=white,scale=0.7}]
  \node            (a) at (0,0.5)  {3};
  \node       [rectangle]     (b) at (1.5,0.5)  {$\bar{2}$};
  \node            (c) at (0,2) {2};
  \node         [rectangle]   (d) at (1.5,2) {$\bar{1}$};

    \draw[->] (c) edge (b);
    \draw[<-] (b) edge (a);
    \draw[->] (a) edge (d);
  \draw[->] (c) edge (d);
\end{tikzpicture}
\caption{$\oDia(w)$, $L(w)$, $L'(w)$, and $G^w$ for $w = 3142$.}\label{fig_toric_ex}
\end{figure}
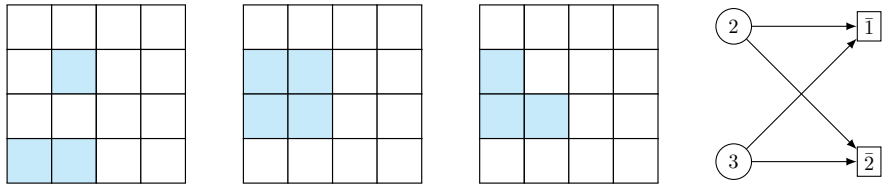
\end{ex}

The following natural question is to study complexity-1 $T\times T$-varieties $Y_{w}$, however the next Theorem shows that there are none.
\begin{thm}\label{nocomplexity1matrixschu}
There are no complexity-1 $T\times T$-varieties $Y_w$.
\end{thm}

We now introduce some notation that will be used in the proof of this theorem.
Suppose that $L(w)$ has connected components $L_1,\ldots, L_k$.
Fix $i$ and let $H_i$ be the set of SW-border of $L_i$, i.e.
\begin{equation}\label{eq_Hi}
     H_i=\{(p,q)\in L_i \mid (p,q-1)\notin L_i,\ (p+1,q)\notin L_i,\ \text{or } (p+1,q-1)\notin L_i\}.
\end{equation}
\Cref{fig_Hi} gives an example of this set.

\begin{figure}[h]
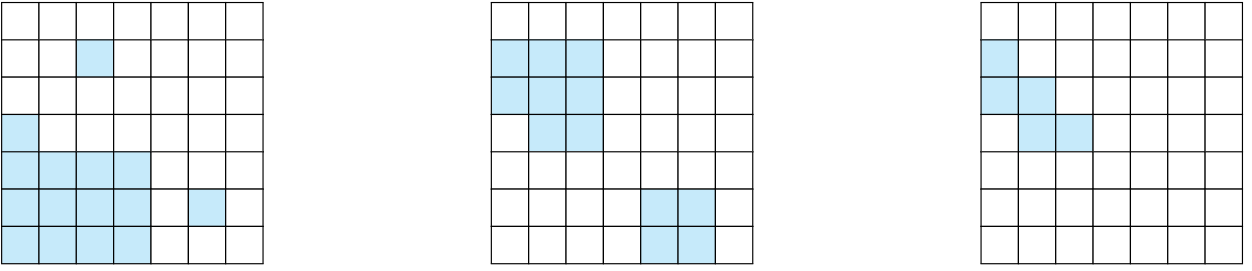

    \centering
    \ytableausetup{smalltableaux, mathmode, boxsize=1.25em}
\begin{ytableau}
*(white) & *(white) & *(white) &*(white) &*(white)&*(white)&*(white) \\
*(white)  & *(white)   &*(cyan!20) &*(white) &*(white)&*(white)&*(white)  \\
*(white) & *(white) & *(white) &*(white) &*(white)&*(white)&*(white) \\
*(cyan!20)    & *(white) & *(white)  &  *(white) &*(white)&*(white)&*(white) \\
*(cyan!20)    & *(cyan!20)    & *(cyan!20)    &*(cyan!20)    &*(white)&*(white)&*(white) \\
*(cyan!20)    & *(cyan!20)    & *(cyan!20)    &*(cyan!20)    &*(white)& *(cyan!20)&*(white) \\
*(cyan!20)    & *(cyan!20)    & *(cyan!20)    &*(cyan!20)    &*(white)&*(white)&*(white) 
\end{ytableau}
\hfill
\begin{ytableau}
*(white) & *(white) & *(white) &*(white) &*(white)&*(white)&*(white) \\
*(cyan!20)    & *(cyan!20)     &*(cyan!20)   &*(white) &*(white)&*(white)&*(white)  \\
*(cyan!20)    & *(cyan!20)    & *(cyan!20)     &  *(white) &*(white)&*(white)&*(white)\\
*(white)  & *(cyan!20)    & *(cyan!20)    &  *(white) &*(white)&*(white)&*(white) \\
*(white) & *(white) & *(white) &*(white) &*(white)&*(white)&*(white) \\
*(white) & *(white) & *(white) &*(white) &*(cyan!20)   &*(cyan!20)   &*(white) \\
*(white) & *(white) & *(white) &*(white) &*(cyan!20)   &*(cyan!20)   &*(white) 
\end{ytableau}
\hfill 
\begin{ytableau}
*(white) & *(white) & *(white) &*(white) &*(white)&*(white)&*(white) \\
*(cyan!20)    & *(white)   &*(white)  &*(white) &*(white)&*(white)&*(white)  \\
*(cyan!20)    & *(cyan!20)    & *(white)   &  *(white) &*(white)&*(white)&*(white)\\
*(white)  & *(cyan!20) & *(cyan!20)    &  *(white) &*(white)&*(white)&*(white) \\
*(white) & *(white) & *(white) &*(white) &*(white)&*(white)&*(white) \\
*(white) & *(white) & *(white) &*(white) &*(white) &*(white) &*(white) \\
*(white) & *(white) & *(white) &*(white) &*(white) &*(white) &*(white) 
\end{ytableau}
\caption{Let $w=3412756$. The leftmost figure is $\oDia(w)$ and the middle one is $L(w)$. Let $L_1$ be the northwesternmost connected component $L(w)$, then $H_1$ is depicted on the right.}\label{fig_Hi}
\end{figure}

\begin{proof}[Proof of \Cref{nocomplexity1matrixschu}]
Suppose that $L(w)$ has connected components $L_1,\ldots, L_k$.
Since $L(w)$ is a skew diagram, each $L_i$ corresponds to a connected component $G_i$ of $G^w$ and therefore $\sigma^w=\sigma_{G_1}^{\vee}\times\cdots\times\sigma_{G_k}^{\vee}$.
We first show that for all $i$, $\dim(\sigma_{G_i}^{\vee}) \le |L_i\cap L'(w)| $.

Fix $i$ and let $H_i$ be as in \eqref{eq_Hi};
we claim that $H_i\cap  \oDia(w)=\varnothing$.
Suppose, by contradiction, that $(p,q)\in H_i\cap  \oDia(w)$.
Then, by the definition of $H_i$ either $p = n$, $q = 1$, or one of $(p, q - 1), (p + 1, q), (p + 1, q - 1)$ is in $\dom(w)$.
 Note that the $p = n$ and $q = 1$ cases can't occur since $(p,q)\in\oDia(w)$ with either $p = n$ or $q = 1$ must also be in $\dom(w)$. 
If at least one of $(p,q-1)$, $(p+1,q)$, or $(p+1,q-1)$ is in $\dom(w)$, then $(p,q)\in\dom(w)$ which is also not possible.
It follows that $H_i\cap  \oDia(w)=\varnothing$, which implies that $H_i\subseteq L_i\cap L'(w)$.

Since, the connected components of $L(w)$ are skew diagrams, the vertices of $G_i$ are the rows and columns containing boxes in $H_i$.
We conclude that  $G_i$ has at most $|H_i|+1$ vertices and, by Lemma~\ref{dimofcone},
    \begin{equation}\label{eq-no-complexity-one-ineq}
    \dim(\sigma^\vee_{G_i})=|H_i|\le |L_i\cap L'(w)|.
    \end{equation}
    
With this equation at hand we proceed to prove that there is no  $Y_w$ of complexity-1.
Assume, by contradiction, that $w$ is such that $\dim(\sigma_w) = |L'(w)| - 1$.
Note that $\dim(\sigma_w)=\dim(\sigma_{G_1}^{\vee})+\cdots+\dim(\sigma_{G_k}^{\vee})$ and $|L'(w)|=|L_1\cap L'(w)|+\cdots+|L_k\cap L'(w)|$.
By \eqref{eq-no-complexity-one-ineq} it follows that there exists $j$ such that $\dim(\sigma_{G_j}^{\vee})=|H_j| = |L_j\cap L'(w)| - 1$ and $\dim(\sigma_{G_i}^{\vee}) = |L_i\cap L'(w)| $ for $i\neq j$.
In particular, we have that 
  \begin{equation}\label{eq-no-complexity-one-connected}
  L_j\cap L'(w)=H_j\cup \{(a,b)\}.
  \end{equation}
Since $(a,b)\in \SW(w)$ there exist $(p,q)\in\oDia(w)$ such that $p\le a$ and $q\ge b$ and, using $(a,b)\notin \oDia(w)$, it follows that either $(a,b+1)\in \SW(w)$ or $(a-1,b)\in \SW(w)$.

Throughout this proof we assume, without loss of generality due to symmetry, that $(a,b+1)\in \SW(w)$.
Suppose $(a, b + 1) \notin \oDia(w)$. Then $(a,b+1) \in L_j\cap L'(w)$ and, by \eqref{eq-no-complexity-one-connected}, $(a,b+1) \in H_j$. 
However, this is not possible since $L_j$ is a skew diagram. 
It follows that $ (a,b+1)\in \oDia(w)$, i.e.
    \begin{equation}\label{eq-no-complexity-one-wlog}
   w(b+1)<a\quad \text{and} \quad w^{-1}(a)>b+1.
    \end{equation}
  
We finish this proof by considering the four possible ways in which $H_j$ and $\{(a,b-1),(a+1,b)\}$ can intersect.

The first case is when $(a,b-1),(a+1,b)\notin H_j$.
From the definition of $\oDia(w)$ we can see that 
    $$
    (a,b-1),(a+1,b)\in \oDia(w),\
    (a,b)\notin \oDia(w)\
    \Rightarrow\
    w(b)=a.
    $$
Since this contradicts \eqref{eq-no-complexity-one-wlog}, we see then that this case is not possible.

The second case is when $(a,b-1),(a+1,b)\in H_j$. 
Since $(a,b-1)\notin \oDia(w)$, by \eqref{eq-no-complexity-one-wlog} it follows that $w(b-1)> a$.
Similarly, $(a,b)\notin \oDia(w)$ and \eqref{eq-no-complexity-one-wlog} imply that $w(b)>a$.
We then have that $a\le n-2$.
Let us show that $(a+2,b-2)\in\dom(w)$.
Since $H_j$ is connected and $(a,b)\notin H_j$, then $(a+1,b-1)\in H_j$. Now, $H_j$ borders part of $\dom(w)$, which is a Young diagram, so $(a+1,b-1),(a+1,b)\in H_j$ imply that $(a+2,b-2)\in\dom(w)$ and, by definition of $\oDia(w)$, that $w(b-1)<a+2$.
We have that $w(b-1)> a$ since \eqref{eq-no-complexity-one-wlog} implies that $w(b-1)\neq a$ and if $w(b-1)<a$ then $(a,b-1)\in\oDia(w)$.
It follows that $w(b-1)=a+1$.

Notice that $(a+2,b-2)\in\dom(w)$ also implies that either $(a+2,b)\in\dom(w)$ or $(a+2,b)\in H_j$.
If $(a+2,b)\in\dom(w)$ then the $1$ in column $b$ of the permutation matrix of $w$ must lie in row $x\le a+1$. 
Since $w(b-1)=a+1$ and by \eqref{eq-no-complexity-one-wlog} this $1$ must be in row $x\le a-1$. 
However, this contradicts that $(a,b)\notin \oDia(w)$ so $(a+2,b)\in \dom(w)$ is not possible.
On the other hand, suppose that $(a+2,b)\in H_j$.
Since $(a,b+1)\in\oDia(w)$, then $(a+1,b+1)\in L_j$.
Furthermore, by \eqref{eq-no-complexity-one-connected}, $(a+1,b+1)\notin L'(w)$ so in fact $(a+1,b+1)\in\oDia(w)$.
However, this would imply that $w^{-1}(a+1)>b+1$ contradicting that $w(b-1)=a+1$.
We conclude that this case is also impossible.

The next case is when $(a,b-1)\in H_j$ and $(a+1,b)\notin H_j$. 
It follows that $(a+1,b)\in\oDia(w)$ and thus $w(b)<a+1$.
However, since $(a,b+1)\in\oDia(w)$ we would have that $(a,b)\in\oDia(w)$, a contradiction.

The last case is when $(a,b-1)\notin H_j$ and $(a+1,b)\in H_j$. 
Since $(a,b)$ is neither in $\oDia(w)$ not in the SW-border of $L_j$, then $(a+1,b-1)\in H_j$.
Now, $w(b)>a+1$ since $w(b)\neq a$, by \eqref{eq-no-complexity-one-wlog}, $w(b)=a+1$ would imply that $(a+1,b-1)\in\oDia(w)$, and $w(b)<a$ would imply that $(a,b)\in\oDia(w)$.
In particular, $a+1<n$.
Since $(a+1,b-1)\in H_j$, we have that $(a,b-1)\notin \dom(w)$, so $(a,b-1)\in L_j$.
By \eqref{eq-no-complexity-one-connected} it follows that $(a,b-1)\in \oDia(w)$. 
Note that
    $$
    (a,b-1)\in \oDia(w),\ (a+1,b-1)\notin \oDia(w)\ \Rightarrow\ (a+1,b+1)\notin \oDia(w)
    $$
and, again by \eqref{eq-no-complexity-one-connected}, we have that $(a+1,b+1)\in H_j$.
Combining this with $a+1<n$ we deduce that $(a+2,b)\in \oDia(w)$.
However, this means that $w(b)<a+2$, contradicting $w(b)>a+1$.

Since none of the cases are possible, it follows that there is no  $Y_w$ of complexity-1.
\end{proof}

\begin{thm}\label{othercomplexitymsv}
There exist $T\times T$-varieties $Y_w$ of complexity-$d$ for $d\geq 2$.
\end{thm}

\begin{proof}
Let $Y_{\alpha}$ be a complexity-$i$ $T\times T$-variety and $Y_{\beta}$ be a complexity-$j$ $T\times T$-variety, for $\alpha \in S_A$ and $\beta \in S_B$. Consider the opposite Rothe diagram of the permutation $\gamma = \alpha_1 \ldots \alpha_A \ \beta_1 + A  \ldots  \beta_{B}+A \in S_{A+B}$:
\[
\begin{tikzpicture}[scale=0.08cm]
\draw[step=0.5cm,color=black] (0,0) grid (1,1);

\node at (+0.25,+0.75) { $D\degree(\alpha)$};
\node at (+0.75,+0.75) {$-$};
\node at (+0.25,+0.25) {$*$   };
\node at (+0.75,+0.25) {$D\degree(\beta)$ };

\end{tikzpicture}
\]
\noindent By the construction the area labelled with $*$ is contained in $\dom(\gamma)$ since there exist no $1$s of the permutation matrix of $\gamma$ in this area. Moreover, $D\degree(w)$ has no boxes in the area labelled with $-$. Since the area $*$ is contained in $\dom(\gamma)$ we obtain the bipartite graph $G^{\gamma}$ is the disjoint union of $G^{\alpha}$ and $G^{\beta}$. Hence we conclude that $Y_{w}$ is a complexity-$(i+j)$ $T\times T$-variety, since $\dim(\sigma_{\gamma} )= \dim(\sigma_{\alpha})+\dim( \sigma_{\beta})$.
Finally, since $Y_{3412}$ has complexity-2, $Y_{4312}$ has complexity-3, and every $d\ge 2$ is in the affine semigroup spanned by $2,3$ the statement follows. 
\end{proof}

The following is a natural question:
\begin{qu}
Given $d\ge 2$, classify the $Y_w$ of complexity-$d$.
\end{qu}
One can observe in $w = 3412$, $w = 42513$ or $w=41523$ that the shapes of the (opposite) Rothe diagrams of $Y_w$ of complexity-$2$ will be more complex than for complexity-$0$.

In this section, we observed that matrix Schubert varieties are not sources of complexity-1 $T\times T$-varieties. Thus, we steer our investigation in the direction of Kazhdan--Lusztig varieties which are generalizations of matrix Schubert varieties and provide many interesting examples of $T$-varieties.

\section{Torus action on Kazhdan--Lusztig varieties}\label{sec: klvarieties}

We now investigate the usual torus action on Kazhdan--Lusztig varieties~$\KL_{v,w}$. Our aim is to study the complexity of this action. To this end we give a combinatorial description of its set of weights, and also of the set of extremal rays of the associated weight cone~$\sigma_{v,w}$. We also discuss the relation between $\sigma_{v,w}$ and the cone of weights of the usual torus action on affine neighborhoods of torus fixed points in Schubert variety $X_w$, i.e.\ $v\Omega_{\id}^\circ \cap X_w$, studied in~\cite{leemasuda}. Finally, in Section~\ref{interactingwithsimpledirectedgraphs} we interpret our results in terms of directed graphs.

%%%%%%%%%%%%%%%%%%%%%%%%%%%%%%%%%%%%%%%%
\subsection{Background on Kazhdan--Lusztig varieties}

Let $G=GL_n(\CC)$, $B$ be the Borel subgroup of upper triangular matrices, $T \subset B$ the maximal torus of diagonal matrices, and $B_-$ the corresponding opposite Borel subgroup of lower triangular matrices. The complete flag variety is $G/B$ and $G$ acts on $G/B$ by left multiplication.
The fixed points of $G/B$ under the left action of $T$ are $wB$ for $w \in S_n$. We have the cell decomposition
    $$G/B = \bigsqcup_{w\in S_n} BwB/B,$$
known as the Bruhat decomposition. 
The closure of the $B$-orbit $BwB/B$ is the \emph{Schubert variety} $X_w\subseteq G/B$, which is a subvariety of dimension $\ell(w)$.
The \emph{opposite Schubert cell} $\Omega_v^\circ$ is the $B_-$-orbit $B_-vB/B$.

\begin{defn}\label{def_KL}
The \emph{Kazhdan--Lusztig variety (KL~variety)} corresponding to $v,w\in S_n$ is
$$\KL_{v,w} := X_w \cap \Omega_v^\circ.$$
\end{defn}
We remark that
that \begin{equation}\label{dimKL}
    \dim (\KL_{v,w}) = \ell(w) - \ell(v),
\end{equation} 
see e.g.~\cite[Cor 3.3]{WY08}.

The nonempty KL varieties are characterized by the Bruhat order.
\begin{defn}\label{def_Bruhat_mat}
The \emph{Bruhat order} is the partial order on $S_n$ defined by $v \leq w$ if $\rank_v(a,b) \leq \rank_w(a,b)$ for all $a,b \in [n]$.
\end{defn}
By~\cite[Lem.~3.10]{fulton}, $\KL_{v,w}$ is non-empty if and only if $v \leq w$. 

A very useful point of view on a KL variety is given by the isomorphism \cite[Lem.~A.4]{kazhdanlusztig}:
\begin{equation}\label{klisomorp}
X_w \cap v \Omega_{\id}^\circ \cong \CC^{\ell(v)} \times \KL_{v,w}.
\end{equation}
The left side of this equation is an affine neighborhood of the $T$-fixed point $vB/B$ in $X_w$.
Therefore, \eqref{klisomorp} tells us that the ideals defining KL~varieties are the ideals defining affine neighborhoods of the $T$-fixed points of Schubert varieties.
As a consequence, we have that KL~varieties are normal.
Indeed, since Schubert varieties are normal, see e.g.~\cite{DCL,RR}, and normality is a local property, it follows that KL~varieties are normal.
Moreover, since KL~varieties are $B$-invariant and $T\subset B$ then they are also $T$-invariant. 
We then have  that KL~varieties are $T$-varieties.\\

Following \cite{WY08,wooyong} we give an explicit description of the coordinate ring $\CC[\KL_{v,w}]$. 
We remark that our conventions for labeling the variables differ from those by Woo-Yong.

\begin{defn}\label{def_omega0v}
Given $v\in S_n$ let $\Sigma_v\subset \CC^{n\times n}$ consist of the matrices $Z$ such that
 \begin{equation}
    \begin{cases}
      Z_{v(i), i} = 1 & \text{for all }i \in [n],  \\
      Z_{v(i), a} = 0  & \text{for } a > i, \\
      Z_{b, i} = 0  & \text{for } b<v(i).
    \end{cases}
  \end{equation}
\end{defn}
\noindent The space $\Sigma_v$ can be realized as follows.
First, write down the 1s of the permutation matrix of $v \in S_n$.
Next set all the entries that are either north or east of a 1 to zero. 
The remaining entries are free. 
Since the free entries are precisely the entries of the opposite Rothe diagram $\oDia(v)$, we have that $\Sigma_v\cong \CC^{|\oDia(v)|}$ and $\CC[\Sigma_v]=\CC[z_{ij}\mid (i,j)\in\oDia(v)]$.
We will denote a generic element of~$\Sigma_v$ by $Z^{(v)}$, i.e.\ a matrix with~0s and~1s at the entries listed above and unknowns $z_{ij}$ at the remaining entries.

\begin{prop}[{\cite[Section 10.2]{fulton-yt}}]
The map $\pi: G \to G/B$ sending a matrix $Z$ to its coset $ZB$ induces a (scheme-theoretic) isomorphism between $\Sigma_v$ and the opposite Schubert cell $\Omega^\circ_v$.
\end{prop}

In \cite{WY08} this isomorphism is used to describe the defining ideal of $\KL_{v,w}$. 
Namely, the following varieties are isomorphic
    $$
    \KL_{v,w}\cong \overline{X_w}\cap \Sigma_v,
    $$
where $\overline{X_w}$ is the matrix Schubert variety corresponding to~$w$.
Thus, the defining ideal of $\KL_{v,w}$ is generated by the determinantal equations obtained by imposing the rank conditions of Fulton's essential set from Theorem~\ref{fultonmatrixschubert} to the generic element in $Z^{(v)}$.

\subsection{Weight cone for a Kazhdan--Lusztig variety}\label{sec: usualtorusaction}

Our aim is to investigate the \emph{usual torus action} on KL~varieties, i.e.\ the restriction of the $T$-action on $G/B$ to $\KL_{v,w}$. In particular we would like to examine its complexity using a combinatorial approach.

The $T$-action on $\Omega^\circ_v$ induces a $T$-action on $\Sigma_v$. Given $M\in T$ and $Z\in\Sigma_v$, let $M\cdot Z=~\pi^{-1}((MZ)B/B)$.
We follow \cite{wooyong} to give a concrete description.
Note that $M\cdot Z\in \Sigma_v$, but in general $MZ \notin \Sigma_v$ because the entries $MZ_{v(i),i}$, $i\in[n]$, need not to be equal 1. 
To obtain $\pi^{-1}((MZ)B/B)\in\Sigma_v$ we multiply $MZ$ on the right by the element of $T$ which make these entries 1; this is accomplished by the diagonal matrix $N$ such that for each $i\in[n]$, $N_{i,i}=(M_{v(i),v(i)})^{-1}$, i.e.\ $N=v^{-1}M^{-1}v$.
This action descends to $\KL_{v,w}$.

We now compute the weights of the $T$-action on $\Sigma_v$.
Let $e_1,\ldots,e_n$ be the standard basis for ${\sf M}(T)_\RR\cong\RR^n$. 
Given $M,N\in T$, as above, and $z_{ij}=(Z^{(v)})_{ij}$ we have
    $$
    (MZ^{(v)}N)_{ij}=M_{ii}N_{jj}z_{ij}=M_{ii}M_{v(j)v(j)}^{-1}z_{ij}.
    $$
 Following the convention that the weights of the $z_{ij}$ are positive roots, as in \cite{wooyong}, the above equation implies that the weight of the variable $z_{ij}$ is $e_{v(j)} - e_{i}$.

\begin{rem}
Matrix Schubert varieties make a special class of KL~varieties \cite[proof of Cor~2.6]{wooyong} and the resulting $T$-action from this section is equivalent to that of \Cref{sec_torus_action_MSch}. 
Concretely, let $w_0^n$ be the longest permutation in $S_n$. Define $w_0 \star w_0 \in S_{2n}$ to be the permutation such that $(w_0 \star w_0)(i) = w_0^n (i)$ and $(w_0 \star w_0)(i+n) = w_0^n (i) + n$, for $i \in [n]$. 
Consider $w$ as a permutation of $S_{2n}$, where $w(i) = i$ for $i \in \{n+1, \ldots , 2n\}$ and let $\hat{w} = w_0^{2n} w \in S_{2n}$. 
Then the matrix Schubert variety $\overline{X_w} \subseteq \CC^{n\times n}$ is isomorphic to the KL~variety $\KL_{ w_0\star w_0,\hat w} \subseteq \CC^{2n\times 2n}$.
 
In \Cref{sec_torus_action_MSch} we considered the $T\times T$-action on $\CC^{n\times n}$ where $z_{ij}$ has weight $e_i - f_j$. On the other hand, $Z^{(w_0 \star w_0)}$ has its free coordinates $z_{ij}$ in rows $\{n+1, \ldots, 2n\}$ and columns $[n]$. The usual torus action on the variable $z_{ij}$ has weight $e_{w_0 \star w_0 (j)} - e_{i}=e_{w_0^n(j)}-e_i=e_{n+1-j} - e_{i}$. Since $\{n+1- j\ | \ j\in[n]\}\cap \{n+1,\ldots,2n\}=\varnothing$, by identifying $e_{i}\leftrightarrow f_{i-n}$ we see that the weights of the two actions are equivalent.
\end{rem}

We denote the \emph{cone of weights for the KL variety} in the lattice $M(T)_{\mathbb{R}} \cong \mathbb{R}^n$ associated to the torus $T$ by $\sigma_{v,w}$. Our aim is to determine the weights of the $T$ action on the KL variety $\KL_{v,w} \cong \overline{X_w} \cap \Sigma_v$, or more precisely the cone $\sigma_{v,w}$. We think of the $T$-action in terms of matrix coordinates. The idea of this section is based on the similarity between the definition of a matrix Schubert variety as a subset of the matrix space $\CC^{n\times n}$, the description~\ref{def_Bruhat_mat} of the Bruhat order of permutations and some well-known facts from representation theory. Since we work with a non-empty KL variety~$\KL_{v,w}$, in what follows we always assume that $v \leq w$ under the Bruhat order.

\subsubsection{Inversions and non-inversions of $v$}

A feature of~$v$ which is very important in understanding $\KL_{v,w}$, is its set of inversions. Recall that $\{i,j\} \subset [n]$, where $i < j$, is an inversion of $v$ if in the one-line notation $j$ stands before~$i$. Or, treating $v$ as a function, $v^{-1}(i) > v^{-1}(j)$. A pair which is not an inversion will be called a \emph{non-inversion} of~$v$. We want to relate (non-)inversions of $v$ to entries of $Z^{(v)}$.

\begin{defn}\label{def_noninv_descr}
We say that the entry $(i,j)$ in $Z^{(v)}$ corresponds to a non-inversion of $v$ if $v(j) < i$ and $\{v(j), i\}$ is a non-inversion of $v$. 
\end{defn}

\begin{rem}\label{rem_all_noninv}
Note that $v(j) < i$ is automatically satisfied for entries which are not set to 0 or 1 by Definition~\ref{def_omega0v}.
Observe that all entries of~$Z^{(v)}$ which are not set to 0 or 1 by Definition~\ref{def_omega0v} correspond to non-inversions of~$v$. This is because these are precisely the entries lying in the opposite Rothe diagram of~$v$, and the condition in Definition~\ref{def_opposite_Rothe} is the same as in Definition~\ref{def_noninv_descr}.
\end{rem}

\subsubsection{Unexpected zeros}

The entries $z_{ij}$ in $Z^{(v)}$ correspond to coordinates on~$\Omega_v^\circ$, and therefore to weights of the $T$-action on~$\Omega_v^\circ$. We need to know which of these weights are still present in the intersection of~$\Sigma_v$ with~$\overline{X_w}$. Some of the weights from the weight cone of $\Omega_v^\circ$ may not lie in the weight cone of $\KL_{v,w}$. The simplest examples are weights corresponding to the entries $z_{ij}$ of the dominant piece for~$w$. These variables vanish on $\KL_{v,w}$, so their weights are not in~$\sigma_{v,w}$, unless they are produced as positive combinations of some other weights, i.e.\ they correspond to monomials.

However, the dominant piece for $w$ is not the only reason for the fact that some weights from $\Omega_v^\circ$ are not present in~$\sigma_{v,w}$. It may happen that the configuration of 0s and 1s in $Z^{(v)}$ together with the conditions for belonging to $\overline{X_w}$ (Fulton's equations coming from submatrix rank conditions, see~Theorem~\ref{fultonmatrixschubert}) make some $z_{ij}$ vanish on the whole~$\KL_{v,w}$. The ideal of $\KL_{v,w}$, i.e.\ the ideal which is obtained from the ideal of $\overline{X_w}$ by substituting 0s and 1s as in Definition~\ref{def_omega0v}, may thus contain some variables $z_{ij}$, not present in the ideal of $\overline{X_w}$.

\begin{defn}\label{def_unexp0}
By an \emph{unexpected 0} (for $\KL_{v,w}$) we understand an entry $z_{ij}$ of $Z^{(v)}$, not constantly~0 on~$\overline{X_w}$, such that for every matrix in $\KL_{v,w}$ we have $z_{ij} = 0$. Equivalently, $z_{ij}$ is in the ideal of $\KL_{v,w}$, but it does not belong to the ideal of~$\overline{X_w}$. 
\end{defn}

In particular, $z_{ij}$ corresponding to entries of the dominant piece of $D^{\circ}(w)$ are unexpected 0s.

\begin{ex}\label{ex_rectangle_unexp0}
Consider $w =47681352$ and $v = 32187654$ in $S_8$. Look at defining conditions for $\overline{X_{w}}$. The diagram $D\degree(w)$ has three parts: the dominant piece $\{z_{41}, z_{51}, z_{61}, z_{71}, z_{81}, z_{82}, z_{83}\}$,  $\{z_{25}, z_{35}\}$ and $\{z_{55}, z_{56}\}$. The entries of the dominant piece are unexpected 0s (marked in blue in the matrix below), but these are not the only ones. The essential boxes of the other pieces of $D\degree(w)$ are $z_{25}$ and $z_{56}$, and they give conditions that the ranks of corresponding lower left corner submatrices are not bigger that 4 and 3 respectively. 

Now look at $v$. In the matrix below, entries in the regions above the horizontal line and to the right of the vertical line are set to~0 or~1 by Definition~\ref{def_omega0v}, hence the only non-constant $z_{ij}$ appear in the lower left rectangle.
When we compute the determinant of a $5\times 5$ minor of $Z^{(v)}$ given by choice of columns $1,\ldots,5$ and rows $2,3,i,7,8$ where $i = 4,5,6$, we obtain $z_{i3}$. The conditions given by $w$ require this determinant to vanish, hence $z_{43}, z_{53}, z_{63}$ are unexpected~0s (marked in red in the matrix below). Also, the determinant of a $4\times 4$ minor of columns $2,4,5$ and rows $5,7,8$ is $z_{52}$, so it is an unexpected~0 (also marked in red).

\[
\left(\begin{array}{ccc|ccccc}
0 & 0 & 1 & 0 & 0 & 0 & 0 & 0\\
0 & 1 & 0 & 0 & 0 & 0 & 0 & 0\\
1 & 0 & 0 & 0 & 0 & 0 & 0 & 0\\
\hline
z_{41} & z_{42} & {\color{red} 0} & 0 & 0 & 0 & 0 & 1\\
{\color{blue} 0} & {\color{red} 0} & {\color{red} 0} & 0 & 0 & 0 & 1 & 0\\
{\color{blue} 0} & z_{62} & {\color{red} 0} & 0 & 0 & 1 & 0 & 0\\
{\color{blue} 0} & z_{72} & z_{73} & 0 & 1 & 0 & 0 & 0\\
{\color{blue} 0} & {\color{blue} 0} & {\color{blue} 0} & 1 & 0 & 0 & 0 & 0\\
\end{array}\right)
\]

Thus we obtain $\KL_{v,w} \cong \mathbb{C}^5$. The weight cone $\sigma_{v,w}$ is spanned by weight corresponding to $z_{ij}$ which are not unexpected 0s, given in the matrix above. Since the weight corresponding to $z_{ij}$ is $e_{v(j)-e_i}$, we have $\cone(e_3-e_4, e_2-e_4, e_2-e_6, e_2-e_7, e_1-e_7)$. It is 5-dimensional, hence the $T$-action on $\KL_{v,w}$ has a dense orbit.

\end{ex}

The main question of the next section is how to determine unexpected 0s for given $v$ and $w$ without computing the ideals. We would like to show a combinatorial characterization of such entries which leads to a direct description of the weight graph and the weight cone of the $T$-action on~$\KL_{v,w}$. We also explain the relation to the description of the weight cone of the $T$-action on $X_w \cap v\Omega_{id}^o$, investigated in~\cite{leemasuda}.

\subsubsection{Weights of the usual torus action}
As explained in \cite{Insko-Yong}, given a subvariety $\mathcal{X}\subseteq G/B$ and $gB\in \mathcal{X}$ the intersection $\mathcal X\cap  g\Omega_{\id}^\circ$
is an affine open neighborhood of $gB$ in $\mathcal X$. 
It follows that $T_{gB}(\mathcal{X}) = T_{gB}(\mathcal X\cap  g\Omega_{\id}^\circ)$. 
If we further assume that $\mathcal{X}$ is a $T$-invariant subvariety of $G/B$ and $gB$ is a $T$-fixed point, then $\mathcal X\cap  g\Omega_{\id}^\circ$ is $T$-invariant.
Moreover, $T_{gB}(\mathcal X\cap  g\Omega_{\id}^\circ)$ inherits a $T$-action and, since $\mathcal X\cap  g\Omega_{\id}^\circ$ is affine, the $T$-weights are the same as those for $\mathcal X\cap  g\Omega_{\id}^\circ$, see e.g.\ \cite[Sec.~5]{Tymoczko}.

The case $\mathcal{X}=X_w$ is well understood and we have that the $T$-weights of $T_{vB}(\mathcal{X}) $ are $\{\epsilon_{v(i)}-\epsilon_{v(j)} \mid i>j,\ t_{v(i)v(j)}v\le w\}$, where $t_{v(i)v(j)}$ is the permutation transposing $v(i)$ and $v(j)$, see e.g.~\cite[Thm.~5.5.3]{BL}.
It follows that these weights are also the $T$-weights of $X_w\cap v\Omega_{\id}^\circ$.
However, these weights come from a choice of coordinates for $\CC[v\Omega_{\id}^\circ]$ different from the one we use.
Next, we describe how to change the weights to our coordinates.
To avoid confusion, let $\epsilon_1,\epsilon_2,\ldots,\epsilon_n$ denote the weights associated to the coordinates from~\cite{BL}.

The content of the next two paragraphs can be found in \cite[Sections 1.2 and 1.3]{brion}.
Let $B_- \subseteq GL_{n}(\CC)$ denote the Borel subgroup of lower triangular matrices and $U_-\subseteq B_-$ the subgroup with diagonal entries equal to 1. 
Note that $U_-\cong\CC^{\binom{n}{2}}$.
We have a $T$-equivariant isomorphism of affine spaces
\begin{equation}\label{eq-u-isom}
 U_-  \cong \Omega_{\id}^\circ, \qquad 
 g\mapsto gB.
\end{equation}
Note that the generic element $Z^{(-)}$ of $U_-$ is the lower triangular matrix with 1s on the diagonal and unknowns $x_{ij}$ for $i>j$.
It follows that $\CC[U_-]\cong\CC[x_{ij} \mid i>j]$.
The $T$-action on $U_-$ is $M\cdot g=MgM^{-1}$, for $M\in T$, so that the weight corresponding to $x_{ij}$ is $\epsilon_i-\epsilon_j$, under the conventions of~\cite{BL}.

The isomorphism \eqref{eq-u-isom} induces the $T$-equivariant isomorphism
\begin{equation}\label{eq-vu-isom}
	vU_-\cong v\Omega_{\id}^\circ, \qquad 
 vg\mapsto vgB.
\end{equation}
The generic element of $vU_-$ is $vZ^{(-)}$ and $\CC[vU_-]\cong\CC[x_{ij} \mid i>j]$.
Note that $x_{ij}$ is the $(v(i),j)$-entry of $vZ^{(-)}$. The $T$-action on $vU_-$ is $M\cdot vg=Mvg(v^{-1}M^{-1}v)$, for $M\in T$.
We then have that $\epsilon_{i}-\epsilon_{v(j)}$ is the weight for the variable at position $(i,j)$, namely $x_{v^{-1}(i)j}$. In fact, \eqref{eq-vu-isom} restricts to a $T$-equivariant isomorphism 
\begin{equation}\label{eq-schu-isom}
	X_w\cap v\Omega_{\id}^\circ\cong vU_-\cap \overline{X_w}.
\end{equation}

By comparing the entries of $vZ^{(-)}$ with the entries of $Z^{(v)}$, we obtain the inclusion $\Sigma_v\hookrightarrow vU_-$ associated to the map
	\begin{align*}
	\CC[x_{ij} \mid i>j]&\rightarrow \CC[z_{ij}\mid (i,j)\in\oDia(v)]
	\\
	x_{v^{-1}(i)j}&\mapsto \begin{cases}z_{ij}  & (i,j)\in\oDia(v)	\\0	& \text{else}\end{cases}
.\end{align*}
Recall that the weight of $z_{ij}$ is $e_{v(j)} - e_{i}$.
It now follows that we can obtain the $T$-weights under our coordinates via the correspondence $\epsilon_i\leftrightarrow -e_i$.

\begin{lem}\label{lem-weights}
The set of $T$-weights of $ \KL_{v,w}$ is $\{e_{v(j)}-e_i \mid t_{v(j),i}v\le w,\ (i,j)\in\oDia(v) \}$. That is, $t_{v(j),i} v \leq w$ if and only if $z_{ij}$ is not an unexpected zero.
\end{lem}

\begin{proof}
By the inclusion $\Sigma_v\hookrightarrow vU_-$ above, we have that the $T$-weights of $\Sigma_v$ are precisely the $T$-weights of $vU_-$ corresponding to the variables in $\Sigma_v$.
These are the weights at positions $(i,j)\in\oDia(v)$, i.e.~the weights
	\begin{align*}
	\{\epsilon_{i}-\epsilon_{v(j)} \mid v^{-1}(i)>j,\ (i,j)\in\oDia(v)\}
	&=
	\{e_{v(j)}-e_i \mid v^{-1}(i)>j,\ (i,j)\in\oDia(v) \}
	\\&=
	\{e_{v(j)}-e_i \mid (i,j)\in\oDia(v) \},
	\end{align*}
where in the last equality we used that, by definition of $\oDia(v)$, if $(i,j)\in\oDia(v)$, then $v^{-1}(i)>j$.

By \cite[Thm.~5.5.3]{BL} and \eqref{eq-schu-isom} we have that the $T$-weights of $vU_-\cap \overline{X_w}$ are
	$$
	\{\epsilon_{i}-\epsilon_{v(j)} \mid v^{-1}(i)>j,\ t_{v(j),i}v\le w\}
	=
	\{e_{v(j)}-e_i \mid v^{-1}(i)>j,\ t_{v(j),i}v\le w\}.
	$$
We obtain the $T$-weights of $\KL_{v,w}$ by using the inclusion
	$$
	\KL_{v,w}\cong\Sigma_v\cap \overline{X_w}\hookrightarrow vU_-\cap \overline{X_w},
	$$
to restrict the preceding weights to those corresponding to the variables in $\Sigma_v$, i.e.
	$$
	\{e_{v(j)}-e_i \mid v^{-1}(i)>j,\ t_{v(j),i}v\le w,\ (i,j)\in\oDia(v)\}
	=
	\{e_{v(j)}-e_i \mid  t_{v(j),i}v\le w,\ (i,j)\in\oDia(v)\}.
	$$
\end{proof}

Now we determine a subset of the set of $T$-weights for $\KL_{v,w}$ which corresponds to the extremal rays of~$\sigma_{v,w}$. This is related to~\cite[Prop.~7.6]{leemasuda} and this is explained in more detail in Remark~\ref{rem-LMgraph}.

\begin{thm}\label{permittingdeletingedges}
The extremal ray generators of the cone $\sigma_{v,w}$ of $T$-weights on $\KL_{v,w}$ are the weights corresponding to $z_{ij}$ from $Z^{(v)}$ such that $t_{v(j),i}v \leq w$ and $\ell(t_{v(j),i}v) - \ell(v) = 1$.
\end{thm}

\begin{proof}
The containment 
	$$\cone(e_{v(j)}-e_i \mid t_{v(j),i}v\le w,\ (i,j)\in\oDia(v),\ \ell(t_{v(j),i}v) - \ell(v) = 1)\subseteq\sigma_{v,w}$$
immediately follows from Lemma~\ref{lem-weights}.

Again by Lemma~\ref{lem-weights}, it suffices to show that if we choose only $z_{ij}$ satisfying  $t_{v(j),i}v\le w$ and in addition $\ell(t_{v(j),i}v) - \ell(v) = 1$, then the corresponding weights do not span a smaller cone. The condition $\ell(t_{v(j),i}v) - \ell(v) = 1$ means that in the one-line notation of~$v$ there are only elements smaller than $v(j)$ or greater than $i$ between $v(j)$ and $i$.
In terms of matrices, this means that in the permutation matrix of $v$ the rectangle determined by $(i,j)$ and $(v(j), v^{-1}(i))$ contains no 1s of~$v$ apart from two corners $(v(j),j)$ and $(i,v^{-1}(i))$. 

First note that, by Remark~\ref{rem_all_noninv}, $z_{ij}$ corresponds to a non-inversion of $v$, hence we have $\ell(t_{v(j),i}v) - \ell(v) > 0$.
Assume that $\ell(t_{v(j),i}v) - \ell(v) > 1$ and let $(v(c),c)$ be an entry containing a~1 of~$v$ in this rectangle. Then the entries $(v(c),j)$ and $(i,c)$ correspond to non-inversions of~$v$.  Moreover, we have $t_{v(j),v(c)}v \leq w$ and $t_{v(c),i}v \leq w$ respectively. Let us prove the latter statement for $(v(c),j)$ and assume that $t_{v(c),i}v \nleq w$. Then there is $(a,b)$ with $v(j) < a \leq v(c)$ and $j \leq b < c$ such that $\rank_{t_{v(c),i}v}(a,b) > \rank_w(a,b)$. However, $\rank_{t_{v(c),i}v}(a,b) = \rank_v(a,b) + 1 = \rank_{t_{v(j),i}}(a,b)$, which contradicts $t_{v(j),i}v \leq w$. Thus we can build a sequence of entries containing~1s of~$v$ 
\[(v(j), j) = (v(c_0),c_0), (v(c_1),c_1),\ldots,(v(c_k),c_k), (v(c_{k+1}), c_{k+1}) = (i, v^{-1}(i))\]
all inside the rectangle with corners $(i,j)$ and $(v(j), v^{-1}(i))$, such that for any $p=0,\ldots,k$ we have $\ell(t_{v(c_p),v(c_{p+1})}v) - \ell(v) = 1$. We construct it inductively starting from $(v(j), j)$ and at the $p$-th step we add $(v(c_p),c_p)$ being the entry containing~1 of~$v$ which is closest to $(v(c_{p-1}),c_{p-1})$ in the rectangle with corners $(v(c_{p-1}),c_{p-1})$ and $(i, v^{-1}(i))$. Then the $T$-weight $e_{v(j)}-e_i$ corresponding to the entry $z_{ij}$ is the sum of $T$-weights $e_{v(c_p)}-e_{v(c_{p+1})}$ for $p=0,\ldots,k$.

We are left with showing that these weights are indeed extremal ray generators of $\sigma_{v,w}$. Let $z_{c_0,d_0}$ satisfy the assumptions of the theorem and the corresponding weight be the sum of weights corresponding to $z_{c_1,d_1}, \ldots, z_{c_k,d_k}$ with $c_1 < c_2 <\ldots < c_k$. That is, $e_{v(d_0)} - e_{c_0} = e_{v(d_1)} - e_{c_1} + \ldots + e_{v(d_k)} - e_{c_k}$.

Note that by Definition~\ref{def_omega0v} we obtain that $$v(d_0) = v(d_1)< c_1 = v(d_2) <c_2 = v(d_3)< \ldots < c_{k-2} = v(d_{k-1}) < c_{k-1} = v(d_k) < c_k =c_0.$$ Then  $(v(d_l),d_l)$ for all $l \in \{2,\ldots,k-1\}$ is contained in the rectangle with corners $(v(d_0),d_0)$ and $(c_0,v^{-1}(c_0))$, which makes $\ell(t_{v(d_0),c_0}) - \ell(v) =1$ impossible.

Since $\ell(t_{v(d_0),c_0}) - \ell(v) =1$ is satisfied by the hypothesis, in the permutation matrix of $v$, the rectangle determined by $(c_0,d_0)$ and $(v(d_0),v^{-1}(c_0))$ must contain no 1s of $v$ other than the corners $(v(d_0),d_0)$ and $(c_0,v^{-1}(c_0))$. However, that is not true, since $(v(d_l),d_l)$ is contained in this rectangle for all $l \in \{2,\ldots,k-1\}$, hence a contradiction.
\end{proof}

\begin{ex}
It is possible that~$z_{ij}$ is an unexpected~0, but the weight corresponding to~$z_{ij}$ belongs to~$\sigma_{v,w}$. Let $v=1324$ and $w=3241$ in $S_4$. Then the entry~$z_{41}$ is an unexpected~0 because it lies in the dominant piece of~$w$. By Lemma~\ref{lem-weights}, the entries $z_{43}$ and $z_{21}$ are not unexpected~0s. By Lemma~\ref{lem-weights}, the weights for those entries are $e_{v(3)} - e_4 = e_2 - e_4$ and $e_{v(1)} - e_2 = e_1 - e_2$ respectively and thus the $T$-weight for the unexpected~0 $z_{41}$ is in~$\sigma_{v,w}$.  
\end{ex}

\begin{rem}\label{rem-LMgraph}
We would like to relate the proofs above to one of the main results of~\cite{leemasuda}, namely the description of the weight cone $D_w(v)$ for the usual torus action on the Schubert variety $X_w$ on an affine neighborhood of the fixed point corresponding to the permutation~$v$ (see~\cite[Def.~7.1, Def.~7.4]{leemasuda}). As in the proof of Lemma~\ref{lem-weights}, one may use~\cite[Thm.~5.5.3]{BL} to investigate $D_w(v)$ by looking at the tangent space to $X_w$ at $v$. One obtains that $D_w(v)$ is spanned by all weights $e_{v(j)} - e_i$ such that $t_{v(j), i}v < w$ and $|\ell(v) - \ell(t_{v(j), i}v)| = 1$, using methods similar to the proof of Theorem~\ref{permittingdeletingedges}. 

Note that this set of weights can be divided into ones corresponding to non-inversions, coming from the KL part in the isomorphism~(\ref{klisomorp}) and having the positive difference of lengths, and the ones corresponding to inversions, coming from the affine part in the isomorphism~(\ref{klisomorp}), with the negative difference of lengths.
\end{rem}

\subsection{Torus action in terms of graphs}\label{interactingwithsimpledirectedgraphs}
In this section, we define the directed acyclic graph $G_{v,w}$ such that its edge cone is the weight cone of the usual torus action on $\KL_{v,w}$. In other words, we interpret Theorem~\ref{permittingdeletingedges} in terms of graphs. Moreover, we also observe the connection to the graph associated to the usual torus action on $v\Omega_{\id}^{\circ} \cap X_w$. We explain how to determine the complexity of KL varieties using these graphs. 

Let $G$ be a directed acyclic graph.
Recall that $ (x \to y)$ denotes a directed edge from vertex $x$ to vertex $y$. A directed edge is called \emph{indecomposable} if there exist no other directed path connecting the vertex $x$ to $y$. Note that in a directed acyclic graph one may decompose edges into indecomposable ones, i.e.\ for any edge $(x \to y)$ there is a path from $x$ to $y$ consisting of indecomposable edges: Either $(x \to y)$ is indecomposable, or we have another directed path from $x$ to $y$. In the latter case, applying this argument inductively to each decomposable edge of this path will end, since otherwise we would have a directed cycle. In particular, remark that the indecomposable edges of $G$ correspond to the extremal ray generators of the dual edge cone $\sigma^{\vee}_G$, defined in Section~\ref{sec-toric-graph}.

\begin{defn}\label{klgraph}
Let $v,w\in S_n$ and let $z_{ij}$ be a coordinate of $Z^{(v)}$. We define the graph $\widetilde{G_{v,w}}$ with
    $$
    V(\widetilde{G_{v,w}}):=[n]\quad\text{and}\quad
    E(\widetilde{G_{v,w}}):=\{(v(j) \to i) \ | \ z_{ij} \text{ is not an unexpected 0 of } \KL_{v,w} \}.
    $$
We define $G_{v,w}$ to be the subgraph with edge set $E(G_{v,w})$ consisting of the indecomposable edges of $E(\widetilde{G_{v,w}})$.
\end{defn}

Recall that $\sigma_{v,w}$ denotes the cone of weights of the $T$-action on the KL variety $\KL_{v,w}$. 
Note that $\sigma_{v,w}=\sigma^{\vee}_{G_{v,w}}=\sigma^\vee_{\widetilde{G_{v,w}}}$.
Moreover, by Theorem~\ref{permittingdeletingedges}, we have that
    \begin{equation}\label{eq-edges-kl-graph}
        E(G_{v,w})=\{ (v(j) \to i) \ | \ i > j,  \ t_{v(j),i} v \leq w \text{ and } \ell(t_{v(j),i} v) - \ell(v) =1 \}.
    \end{equation}

Thus, we can use the following formulas for the dimension of the dual edge cone of $\sigma_{v,w}$ of $G_{v,w}$ and relate it to the complexity of a KL variety.

\begin{cor}\label{cor_dim_formulas}
The dimension of $\sigma_{v,w}$ of $G_{v,w}$, calculated as in \Cref{dimofcone}, is:
\begin{equation}\label{eq-KL-cone-dim} \dim(\sigma_{v,w}) =|V(G_{v,w})| - \# \text{(connected components of } G_{v,w}).
\end{equation}

Moreover, $\KL_{v,w}$ is a KL variety of complexity-$k$  if and only if 
    \begin{equation}\label{eq-KL-complexity}
    \dim(\sigma_{v,w}) = \dim(\KL_{v,w}) - k = \ell(w) - \ell(v)-k.
    \end{equation}
\end{cor}

We now compare the complexity of KL varieties with the complexity of a Richardson variety. Given $v,w\in S_n$, the \emph{Richardson variety} $X^v_w$ is defined to be the intersection of the Schubert variety $X_w$ and the opposite Schubert variety $X^v: = \overline{\Omega^\circ_v}=\overline{B_-vB/B}$. 
Since Richardson varieties are invariant under the $T$-action, they are $T$-varieties.

\begin{cor}
\label{cor_richi_KL} The complexity of the Kazhdan–Lusztig variety $\KL_{v,w}$ is the same as that of the Richardson variety $X_w^v$.
\end{cor}

\begin{proof}
The Bruhat interval polytope ${\sf Q}_{v,w}$, introduced by Kodama--Williams in \cite{KodWil}, is the moment map image of $X_w^v$ with respect to the $T$-action. Because of this, the complexity of $X_w^v$ is given by
    $$
    \dim(X_w^v)-\dim({\sf Q}_{v,w})=\ell(w)-\ell(v)-\dim({\sf Q}_{v,w}).
    $$
In \cite[Thm.~4.6]{TsuWil} a formula for $\dim({\sf Q}_{v,w})$ is given and using \cite[Def.~4.9 and Prop.~4.10]{TsuWil} it can be phrased as counting the number of connected components of the graph with vertex set $[n]$ and edge set
$$
\overline{T}(v,w)=\{(i \to j) \mid i<j,\ vt_{ij}\le w,\ \ell(vt_{ij})-\ell(v)=1 \}.
$$
One can see from \eqref{eq-edges-kl-graph} that this graph is isomorphic to $G_{v,w}$. It follows from \eqref{eq-KL-cone-dim} that 
$$
\dim({\sf Q}_{v,w})=\dim(\sigma_{v,w}),
$$
and thus both varieties have the same complexity.
\end{proof}

As a consequence of this corollary and \cite[Prop.~6.4]{LeeMasPar_toric_bruhat}, we immediately have that for every $k$ there is a KL variety of complexity-$k$. We will also verify this in the next section, when we compute the complexity of $\KL_{v,w}$ for $v$ of small length.

\begin{rem}
\Cref{cor_richi_KL} also has a geometric explanation. Fix $v\le w$.
In \cite[proof of Lemma 2.1]{KnutsonWooYong} it is shown that for each $u$ such that $v\le u\le w$ there is a $T$-equivariant isomorphism
$$
u\Omega_{\id}^\circ \cap X_w^v\to(X_w\cap \Omega_{u}^\circ)\times (X^v\cap \Omega^{u}_\circ),
$$
where $\Omega^{u}_\circ$ is the $B$-orbit $BuB/B$.
The case $u=v$ gives a $T$-equivariant isomorphism
$$
v\Omega_{\id}^\circ \cap X_w^v\to\KL_{v,w}\times \{vB\}\cong\KL_{v,w}
.$$
As a consequence, the cone spanned by the $T$-weights of $T_{vB}(X_w^v)$ is linearly isomorphic to the cone spanned by the $T$-weights of $\KL_{v,w}$. 
Since the variety on the left is an affine open neighborhood of $vB$ in $X_w^v$, \cite[Sec.~5]{Tymoczko} implies that the cone spanned by the $T$-weights of $T_{vB}(X_w^v)$ is equal to the
cone spanned by the edges of the moment polytope ${\sf Q}_{v,w}$ incident to the vertex $v$.
Now, the dimension of this cone is $\dim({\sf Q}_{v,w})$ and \Cref{cor_richi_KL} follows.
\end{rem}

Although we work with $G_{v,w}$ for the torus action on $\KL_{v,w}$, it is worth discussing another graph which is used to understand the torus action on $v\Omega_{\id}^\circ \cap X_w$ in \cite{leemasuda}. It is the graph related to the cone $D_w(v)$ of $T$-weights as in Remark~\ref{rem-LMgraph}. This graph coincides with~$G_{v,w}$ when $v=\id$ and we observe that in general $G_{v,w}$ can be obtained by deleting some edges from this graph as explained in Example~\ref{deletion}.

\begin{defn}\cite[Def.~7.1]{leemasuda}
We define the edge set
$$\widetilde{E}_w(v):= \{ (v(i) \to v(j)) \ | \ 1 \leq i < j \leq n, \ t_{v(i),v(j)} v \leq w \text{ and } |\ell(v) - \ell(t_{v(i),v(j)} v)| =1 \}.$$
and denote the set of indecomposable edges of $\widetilde{E}_w(v)$ by $E_w(v)$. We associate graphs $\Gamma_w(v)$ and $\widetilde{\Gamma}_w(v)$ to $E_w(v)$ and $\widetilde{E}_w(v)$, respectively.
\end{defn}

\begin{ex}\label{complex1KLex}
Let $w = 245163 \in S_6$ and $v= \id = 123456 \in S_6$. The dimension of the KL variety $\KL_{v,w} \cong \id \Omega_{\id}^\circ \cap X_w$ is $\ell(w)-\ell(v) = 6-0 =6$. The essential set of $D\degree(w)$ consists of the boxes labeled by $(3,1)$, $(6,4)$, $(5,2)$ and $(3,4)$. Therefore after imposing the Fulton conditions on $Z^{(v)}$, we obtain the second matrix

\[
\begin{bmatrix}
    1       & 0 & 0 & 0  & 0 & 0\\
    z_{21} & 1 & 0 & 0  &0 & 0\\
   z_{31}  & z_{32} & 1 & 0& 0 & 0  \\
  z_{41}       & z_{42} & z_{43} & 1 & 0 & 0 \\
  z_{51} & z_{52} & z_{53}& z_{54}& 1& 0 \\
   z_{61} & z_{62} & z_{63}& z_{64}& z_{65}&1 
\end{bmatrix}
\qquad
\begin{bmatrix}
    1       & 0 & 0 & 0  & 0 & 0\\
    z_{21} & 1 & 0 & 0  &0 & 0\\
  0   & z_{32} & 1 & 0& 0 & 0  \\
  0       & z_{42} & z_{43} & 1 & 0 & 0 \\
  0 & 0 & z_{53}& z_{54}& 1& 0 \\
   0 & 0 & 0 & 0& z_{65}&1 
\end{bmatrix}
\] 

where $ z_{42} z_{54} + z_{32} z_{53} - z_{54} z_{43} z_{32}=0$. The graphs $\widetilde{G_{v,w}}$ and $\widetilde{\Gamma}_w(v)$ are the same, since $\id$ has no inversions. The dimension of $\sigma_{v,w}$ is 5 and  $\KL_{v,w}$ is a a $T$-variety of complexity-1.

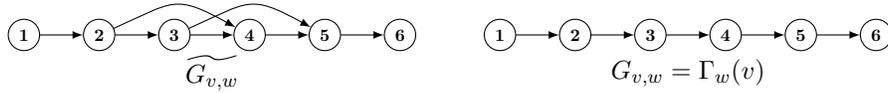
\begin{center}
\begin{tikzpicture}[baseline=1,scale=1,every path/.style={>=latex},every node/.style={draw,circle,fill=white,scale=0.6}]

  \node           (b) at (2,0)  {\bf{3}};
  \node           (f) at (1,0)  {\bf{2}};
  \node           (g) at (0,0)  {\bf{1}};
  \node      (c) at (4,0) {\bf{5}};
  \node        (d) at (3,0) {\bf{4}};
  \node         (h) at (5,0)  {\bf{6}};
  \node   (l) [fill=none,draw=none,scale=1.5] at (2.5,-0.5)  {$\widetilde{G_{v,w}}$};

 \draw[->] (g) edge (f);
  \draw[->] (f) edge (b);
   \draw[->] (b) edge (d);
 \draw[<-] (c) edge (d);
  \draw[<-] (h) edge (c);

 \draw[->] (f) .. controls (2,0.45) and (2,0.55) .. (d);

 \draw[->] (b) .. controls (3,0.45) and (3,0.55) .. (c);

\end{tikzpicture}
\qquad
\begin{tikzpicture}[baseline=1,scale=1,every path/.style={>=latex},every node/.style={draw,circle,fill=white,scale=0.6}]

  \node           (b) at (2,0)  {\bf{3}};
  \node           (f) at (1,0)  {\bf{2}};
  \node           (g) at (0,0)  {\bf{1}};
  \node      (c) at (4,0) {\bf{5}};
  \node        (d) at (3,0) {\bf{4}};
  \node         (h) at (5,0)  {\bf{6}};
    \node   (l) [fill=none,draw=none,scale=1.5] at (2.5,-0.5)  {$G_{v,w}=\Gamma_w(v)$};

 \draw[->] (g) edge (f);
  \draw[->] (f) edge (b);
   \draw[->] (b) edge (d);
 \draw[<-] (c) edge (d);
  \draw[<-] (h) edge (c);

\end{tikzpicture}
\end{center}
\end{ex}

\begin{rem}\label{rem_delete_edges}
Our method of drawing the graph associated to the weight cone is to place the vertices, corresponding to coordinates on~$T$, on a line, numbered $1,\ldots,n$ from left to right. A~useful observation is that a weight corresponding to a non-inversion, i.e.\ coming from the KL variety $\KL_{v,w}$, is a right pointing arrow, and a weight corresponding to an inversion, i.e.\ coming from the affine part $\mathbb{C}^{\ell(v)}$ in the isomorphism~(\ref{klisomorp}), is a left pointing arrow.
Thus, to obtain $G_{v,w}$ from $\widetilde{\Gamma}_w(v)$, we delete all left pointing arrows, however this does not work if one deletes edges from ${\Gamma}_w(v)$ to obtain $G_{v,w}$. Example~\ref{deletion} explains why one should be careful with this operation.
\end{rem}

The following is an interesting fact which motivated us to define $G_{v,w}$ with the result of Theorem~\ref{permittingdeletingedges} instead of deleting edges from $\widetilde{\Gamma}_w(v)$ or $\Gamma_w(v)$, where one needs to be careful not to loose the edges corresponding to $G_{v,w}$. 

\begin{prop}\label{decomposableparts}
Every decomposition of an edge of $\widetilde{E}_w(v)$ contains at least one edge from $E(\widetilde{G_{v,w}})$ and at least one edge from $\widetilde{E}_w(v)\setminus E(\widetilde{G_{v,w}})$. 
\end{prop}

\begin{proof}
 Suppose that there exists a decomposable edge $ (v(a) \to v(c)) \in \widetilde{E}_w(v)$ where $(a,c)$ is an inversion of $v$, i.e.\ $a<c$, $v(a)>v(c)$. Without loss of generality assume that $ (v(a) \to v(c)) = (v(a) \to v(b))  +  (v(b) \to  v(c))$. If both edges $ (v(a) \to v(b))$ and $(v(b) \to v(c))$ in $\widetilde{E}_w(v)$ correspond to the torus action on $\CC^{\ell(v)}$, then $v(b) \in [v(c), v(a)]$ which contradicts with the condition $|\ell(v) - \ell(t_{v(i),v(j)} v)| =1$. For the case where $ (v(a) \to v(c)) \in \widetilde{E}_w(v)$ corresponds to the torus action on $\KL_{v,w}$, the same argument follows.    
\end{proof}

\begin{ex} \label{deletion}
Consider the KL variety $\KL_{v,w}$ with $v=s_4=12354$ and $w=13542$. By Definition~\ref{klgraph}, one obtains $E(\widetilde{G_{v,w}})=E(G_{v,w})$ as in the figure.

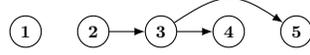
\begin{center}
\begin{tikzpicture}[scale=1.8,every path/.style={>=latex},every node/.style={draw,circle,fill=white,scale=0.6}]

  \node           (b) at (2,0)  {\bf{5}};
  \node           (f) at (1,0)  {\bf{3}};
  \node           (g) at (0,0)  {\bf{1}};
  \node         (c) at (1.5,0) {\bf{4}};
  \node       (d) at (0.5,0) {\bf{2}};

    \draw[<-] (c) edge (f);
    \draw[->]  (f).. controls (1.45,0.3) and (1.55,0.3) .. (b);

  \draw[<-] (f) edge (d);

\end{tikzpicture}

\end{center}
In particular, this can be also achieved by deleting the edge $ (5 \to 4)$ from the edge set $\widetilde{E_w}(v)$, which is the edge corresponding to the torus action on $\CC^1$. However, working with $\Gamma_w(v)$ causes loosing the edge $(3 \to 4)$ which corresponds to $z_{43}$ of $Z^{(v)}$.
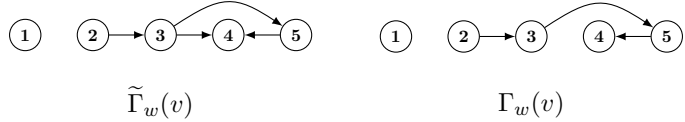
\begin{center}
\begin{tikzpicture}[scale=1.8,every path/.style={>=latex},every node/.style={draw,circle,fill=white,scale=0.6}]

  \node           (b) at (2,0)  {\bf{5}};
  \node           (f) at (1,0)  {\bf{3}};
    \node           (fi) [draw=none, fill =none, scale =1.5] at (1,-0.5)  {$\widetilde{\Gamma}_w(v)$};
  \node           (g) at (0,0)  {\bf{1}};

  \node         (c) at (1.5,0) {\bf{4}};
  \node       (d) at (0.5,0) {\bf{2}};

 \draw[<-] (c) edge (b);
    \draw[<-] (c) edge (f);
    \draw[->]  (f).. controls (1.45,0.3) and (1.55,0.3) .. (b);

  \draw[<-] (f) edge (d);

\end{tikzpicture}
\qquad
\begin{tikzpicture}[scale=1.8,every path/.style={>=latex},every node/.style={draw,circle,fill=white,scale=0.6}]

  \node           (b) at (2,0)  {\bf{5}};
  \node           (f) at (1,0)  {\bf{3}};
  \node           (g) at (0,0)  {\bf{1}};
  \node           (fi) [draw=none, fill =none, scale =1.5] at (1,-0.5)  {$\Gamma_w(v)$};
  \node         (c) at (1.5,0) {\bf{4}};
  \node       (d) at (0.5,0) {\bf{2}};

 \draw[<-] (c) edge (b);

    \draw[->]  (f).. controls (1.45,0.3) and (1.55,0.3) .. (b);

  \draw[<-] (f) edge (d);

\end{tikzpicture}
\end{center}

\end{ex}

To summarize, the aim of this section is to develop combinatorial which can be used to compute the complexity of the usual torus action on a KL variety $\KL_{v,w}$. This is realized in particular in Lemma~\ref{lem-weights} and Theorem~\ref{permittingdeletingedges}, which describe the weight cone $\sigma_{v,w}$, and Corollary~\ref{cor_dim_formulas}, which relates the complexity of considered action to the structure of the weight cone. In Section~\ref{lowcomplexityexamples}, we approach the problem of finding and studying low-complexity KL varieties using the results of this section and also different methods related to the properties of permutations.
An application of results of Section~\ref{sec: klvarieties} to determining the complexity of KL varieties and producing interesting sets of low-complexity examples will also be presented in a forthcoming paper of the authors.

%------------------------------
\section{The complexity of KL varieties for $v$ of small length}\label{lowcomplexityexamples}

In this section, we investigate certain low complexity $T$-variety examples for $\KL_{v,w}$ with respect to the usual torus action. Throughout this section, a different approach with reduced word expressions of permutations is introduced. First, we describe some cases where $\KL_{v,w}$ is toric. After that, we consider the cases in which $v$ has length 0 or 1, and study the complexity in terms of the shape of a reduced word expression of $w$ in Theorems~\ref{visid} and \ref{vistransp}. This section ends with motivating observations for other forms of $v$, in particular we explain why studying the complexity for a general $v$ and $w$ using reduced word expressions is challenging.

%------------------------------
\subsection{Toric KL varieties}

We begin with a condition on $\widetilde{G_{v,w}}$ that guarantees that $\KL_{v,w}$ is toric.
\begin{prop}\label{smoothtoricKL}
The KL~variety $\KL_{v,w}$ is toric, if $\widetilde{G_{v,w}}$ is a forest. In this case, $\KL_{v,w}$ is an affine space.
\end{prop}
\begin{proof}
If $\widetilde{G_{v,w}}$ is a forest, then the dimension of $\sigma_{v,w}$ is the number of edges of $\widetilde{G_{v,w}}$, which is equal to the number of entries in $Z^{(v)}$ that are not unexpected zeros. It follows that $\dim(\sigma_{v,w})= \dim(\Sigma_v)\ge \dim(\KL_{v,w})$ and hence these have to be equal and $\KL_{v,w}$ is the toric variety of $\sigma_{v,w}$. By Theorem~\ref{thm-primitive}, the ideal defining this toric variety is the zero ideal.
We conclude that $\KL_{v,w}$ is an affine space.
\end{proof}

\begin{ex}\label{toricKL}
Consider the KL~variety $\KL_{v,w}\cong \CC^{5}$ from  Example~\ref{ex_rectangle_unexp0} with $w =47681352\in S_8$ and $v = 32187654\in S_8$. The graph $G_{v,w}$=$\widetilde{G_{v,w}}$ is a tree with five edges. 
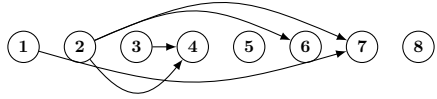
\begin{center}
\begin{tikzpicture}[scale=1.5,every path/.style={>=latex},every node/.style={draw,circle,fill=white,scale=0.6}]

  \node           (b) at (1,0)  {\bf{3}};
  \node           (f) at (0.5,0)  {\bf{2}};
  \node           (g) at (0,0)  {\bf{1}};
  \node        (c) at (2.5,0) {\bf{6}};
  \node       (i) at (2,0) {\bf{5}};
  \node       (d) at (1.5,0) {\bf{4}};
  \node          (h) at (3,0)  {\bf{7}};
  \node         (j) at (3.5,0) {\bf{8}};

  \draw[->] (b) edge (d);
  \draw[->] (f) .. controls (0.9,-0.5) and (1.1,-0.5) ..  (d);

  \draw[->] (f) .. controls (1.65,0.5) and (1.85,0.5) ..  (h);
  \draw[->] (f) .. controls (1.4,0.4) and (1.6,0.4) ..  (c);
  \draw[->] (g) .. controls (1.4,-0.4) and (1.6,-0.4) ..  (h);

\end{tikzpicture}

\end{center}
The dimension of the edge cone of $G_{v,w}$ is five and therefore $\KL_{v,w}$ is a toric variety.
\end{ex}

\noindent  Note that the other direction of Proposition~\ref{smoothtoricKL} is not true. 
\begin{ex}
Let $v = 2143 \in S_4$ and $w=4231 \in S_4$. The defining ideal of the KL~variety $\KL_{v,w}$ is generated by $z_{31}z_{42} - z_{32} z_{41}$ and the dimension of the variety is~3. On the other hand, the associated graph $G_{v,w}$ is $K_{2,2}$ with directed edges. Hence~$\KL_{v,w}$ is a toric variety but~$\widetilde{G_{v,w}}$ is not a forest.
\end{ex}

We remark that \cite[Prop.~4.12]{TsuWil} gives necessary and sufficient conditions for a Richardson variety to be toric in terms of a different graph. 
By \Cref{cor_richi_KL}, these conditions can be applied to characterize toric KL varieties.
This graph, defined in \cite[Def.~4.5]{TsuWil}, is more complicated than $\widetilde{G_{v,w}}$ since it requires giving a maximal chain $x_0\lessdot \cdots \lessdot x_m$ in the interval $[v,w]$ and computing $x_i^{-1}x_{i+1}$ for all $i$.
Also, \cite[Thm.~1.1]{LeeMasPar_toric_bruhat} gives a different criterion for a Richardson variety to be toric in terms of the face structure of the Bruhat interval polytope ${\sf Q}_{v,w}$.

For the remainder of this section, we look at the weight cone $\sigma_{v,w}$ of $\KL_{v,w}$ from a different point of view. Namely, in terms of reduced word expressions, by using some classical results on Bruhat order.
\begin{defn}
A \emph{simple reflection} $s_i \in S_n$ is the adjacent transposition $t_{i,i+1}$. Each permutation in $S_n$ can be written as a product of simple reflections. If the length of this product i.e.\ $\ell(w)$ is minimum among all such expressions, this product is called a \emph{reduced word expression}. 
\end{defn}

\begin{prop}\label{distinctsimplereflections}
If $w \in S_n$ is a product of distinct simple reflections, then the KL~variety $\KL_{v,w}$ is toric for all $v \leq w$.
\end{prop}
\begin{proof}
By \cite[Thm.\ 2 and Thm.\ 4]{Karuppuchamy}, $X_w$ is a toric variety and hence $v\Omega_{id}^\circ\cap X_w \cong \KL_{v,w} \times \CC^{\ell(v)}$ is the affine toric variety corresponding an affine neighborhood of $T$-fixed point $vB$ in $X_w$. This also corresponds to a maximal cone (edge cone of $\widetilde{\Gamma}_w(v)$) of the normal fan of $X_w$ for all $v\leq w$. By Proposition~\ref{decomposableparts}, since every decomposition of an edge of $\widetilde{E}_w(v)$ contains at least one edge from $E(\widetilde{G_{v,w}})$ and at least one edge from $\widetilde{E}_w(v)\setminus E(\widetilde{G_{v,w}})$, the torus action $T$ on $v\Omega_{id}^\circ\cap X_w$ factors through two torus actions on $\KL_{v,w}$ and $\CC^{\ell(v)}$. This concludes the proof. 
\end{proof}
Note that $\KL_{v,w}$ being toric does not always imply that $X_w$ is also toric. 

\begin{ex} \label{length1example}
Consider the KL~variety $\KL_{v,w}$ with $v=s_4=12354$ and $w=13542$. 
In this case we have that $\widetilde{G_{v,w}}$ is a tree and thus $\KL_{v,w}$ is toric. On the other hand, $X_w$ is not toric because $w=s_2 s_4 s_3 s_4$, i.e.~it is not the product of distinct simple reflections, see \cite{Karuppuchamy}.
\end{ex}

Our next approach is to fix $v \in S_n$ and change $w \in S_n$ according to its length. The main method for our proofs is motivated by the following theorem, called the \emph{subword property}. 

\begin{thm}\cite[Thm.~2.2.2]{bjoernerbrenti}\label{subwordproperty}
Let $w = s_{i_1} \ldots s_{i_q}$ be a reduced word expression of $w$. Then, $u \leq w$ if and only if there exists a reduced word expression of $u=s_{i_{j_1}} \ldots s_{i_{j_k}}$ such that $1 \leq j_1 < \ldots < j_k \leq q$. 

\end{thm}

%------------------------------
\subsection{Case $v=id$}
Let us start with the easiest case $\ell(v)=0$, i.e.\ $v=\id$. As we observed in Example~\ref{complex1KLex}, even in this case one can obtain interesting cases. Moreover, studying $X_w$ locally at the fixed point $\id B$ can be accomplished by studying the KL~variety $\KL_{\id,w}$. Notice that for $v=\id$ the condition $v \leq w$ holds for any $w\in S_n$.

\begin{thm}\label{visid}
The usual torus action on the KL~variety $\KL_{\id,w}$ is of complexity-$k$ if and only if 
    $$
    |\{s_i \mid s_i \le w\}|
    =\ell(w)-k.
    $$
\end{thm}

\begin{proof}
The KL~variety $\KL_{\id,w}$ has dimension $\ell(w)$. By \eqref{eq-edges-kl-graph}, the graph for the usual torus action on $\KL_{\id,w}$ has edges
    $E(G_{\id,w}) = \{ (a \to a+1) \ | \ t_{a, a+1} \id \leq w\}$
and $G_{\id,w}$ is a forest. 
It follows that $\dim(\sigma_{v,w})=|E(G_{\id,w})|$.
By the subword property, $ (a\to a+1)\in E(G_{\id,w})$ if and only if there exists a reduced expression for $w$ containing $s_a$.
Therefore, $\dim(\sigma_{v,w})$ equals the number of distinct simple reflections in a reduced expression of $w$. Since the dimension of the KL~variety is $\ell(w)$, this gives a complexity-$k$ usual torus action if and only if this number is $\ell(w)-k$. 
\end{proof}

\begin{ex}\label{ex: not distinct transp}
Consider Example~\ref{complex1KLex}, where $v=\id$ and $w=245163$. A reduced expression of $w$ is $s_1 s_3 s_2 s_4 s_3 s_5$. It consists of all simple reflections with $s_3$ repeated once. Hence $\KL_{v,w}$ admits a complexity-1 torus action. Moreover, $G_{\id,w}$ is a tree with 6 vertices and thus $\KL_{v,w}$ is an affine space.  
\end{ex}

As a corollary of this result, we can characterize the $w$ for which $\KL_{\id,w}$ is toric. 

\begin{cor}
$\KL_{\id,w}$ is toric if and only if $w$ is a product of distinct simple reflections. Moreover, in this case $\KL_{\id,w}$ is an affine space.
\end{cor}

\begin{proof}
Proposition~\ref{distinctsimplereflections} shows that $\KL_{\id,w}$ is toric if $w$ is a product of distinct simple reflections.
Now suppose that $\KL_{\id,w}$ is toric.
The fact that $w$ is a product of distinct simple reflections immediately follows from Theorem~\ref{visid}, since
    $$|\{s_i \mid s_i \le w\}|
    =\ell(w)
    .$$
Note that this implies also that $G_{\id,w}$ is a forest.
It follows from Theorem~\ref{thm-primitive} that $\KL_{\id,w}$ is an affine space.
\end{proof}

This corollary is recovering a known result. The Richardson variety $X_w^{\id}$ is precisely the Schubert variety $X_w$ so by \cite[Thm.~2 and Thm.~4]{Karuppuchamy} $X_w=X_w^{\id}$ is a toric variety if and only if $w$ is a product of simple reflections.
Applying \Cref{cor_richi_KL} this is equivalent to the corollary above. We remark that \cite[Thm.~1.1]{leemasudapark} has other equivalent conditions characterizing toric Schubert varieties.
Also, \cite[Thm.~1.2 and Thm.~1.3]{leemasudapark} give equivalent conditions characterizing complexity-1 Schubert varieties.

The permutation $w=245163$ in Example~\ref{ex: not distinct transp} is a so-called irreducible permutation and it turns out to be interesting to study this class of permutations.

A permutation $w \in S_n$ is called \emph{reducible} if there is $j<n$ such that 
    \begin{equation}\label{eq-reducible}
    \{w(1),\ldots,w(j)\} = \{1,\ldots,j\},
    \end{equation}
that is, the one-line notation for $w$ starts with a permutation of $\{1,\ldots,j\}$. Otherwise $w$ is called \emph{irreducible}. These permutations were introduced by Comtet in \cite{comtet} and sometimes called indecomposable as well.

\begin{prop}\label{irreducibleperm}
Let $v = \id$. Then $w$ is irreducible if and only if $G_{v,w}$ is connected, i.e.\ the dimension of $\sigma_{v,w}$ is $n-1$, as large as possible.
\end{prop}

\begin{proof}
Suppose that $w$ is irreducible.
It suffices to show that for $i=2,\ldots,n$ the variable $z_{i,i-1}$ is not an unexpected 0; equivalently, each $(i-1 \to i)$ is an edge of $G_{v,w}$. We check the conditions of \eqref{eq-edges-kl-graph}.  
We have $t_{i-1,i} v = t_{i-1,i}=s_{i-1}$.
Since $\rank_{s_i}(a,b)=0$ for all $a> b$ such that $(a,b)\neq (i,i-1)$, it follows from Definition~\ref{def_Bruhat_mat} that $s_{i-1}\le w$ if and only if $\rank_w(i,i-1)\geq 1$.
Note that if $\rank_w(i,i-1)= 0$, then $w$ would not map any element of $\{1,\ldots,i-1\}$ to an element of $\{i,\ldots,n\}$, contradicting that $w$ is irreducible. 
We conclude that $(i-1 \to i)$ is an edge of $G_{v,w}$. 

Next, assume that $w$ is reducible and let $j$ be as in \eqref{eq-reducible}.
Consider the descent set of $w$ $$\text{des} (w): = \{i\in[n] \ | \  w(i) > w(i+1)\}.$$ Every reduced word expression of $w$ is found by pulling out $s_i$ for $i \in \text{des}(w)$ and replacing $w$ with $s_i w$ to perform the same step until obtaining $\id$. 
Since $w$ is reducible as described, $w(j) < w(j+j')$ for all $j' \in [n-j]$ and this property is preserved while constructing a reduced word expression for $w$. 
Thus $s_j$ is not a factor of a reduced word expression of $w$. This means that $(j \to j+1) \notin E(G_{v,w})$ and $s_j \nleq w$. Assume that there exists an edge $(k \to k') \in E(G_{v,w})$ with $k \in [j]$ and $k' \in \{j+1, \ldots, n\}$. This means that $s_j \leq t_{k,k'} \leq w$, which is a contradiction.  
\end{proof}

\begin{rem}\label{obsirreducible}
The proof of Proposition~\ref{irreducibleperm} implies that $w$ is irreducible if and only if $s_i$ appears in a reduced word expression of $w$ for all $i \in [n-1]$. 
\end{rem}

We end the case $v=\id$ with the following remark which gives an equivariant decomposition of $\KL_{\id,w}$ whenever $w\in S_{n_1}\times\cdots\times S_{n_m}$.

\begin{rem}
Let $w_1 \in S_{n_1},\ldots, w_m \in S_{n_m}$.
Given $w=w_{1} \times \ldots \times w_{m}\in S_{n_1}\times\cdots\times S_{n_m}$, the structure of $\oDia(w)$ tells us that if $z_{ij}$ is not an unexpected zero of $\KL_{\id,w}$, then 
    $$
    (i,j)\in \bigcup_{a=1}^m\left[\sum_{b=1}^{a-1} n_{b}+1,\sum_{b=1}^{a} n_{b}\right]^2.$$
That is, the permutation matrix of $w$ contains permutation matrices of $w_1,\ldots,w_m$ placed at the diagonal, and $(i,j)$ lies in one of these matrices.
This leads to a $T$-invariant decomposition of $\KL_{\id,w} = \KL_{\id,w_{1}}\times \cdots \times \KL_{\id,w_{m}}$, and to the decomposition $T = T_1\times \ldots \times T_m$ such that the action of~$T$ on~$\KL_{\id,w_{j}}$ factors through~$T_j$.
\end{rem}

\subsection{Case $v=s_a$}
In this section we investigate the complexity of the torus action when $v$ is a permutation of length one.
Note that in this setting, $\KL_{v,w}$ slightly differs from $s_a\Omega_{\id}^\circ \cap X_w$ by an affine part, which is a line. We can see in the next theorem that by increasing the dimension of $v$ by one, the difficulty of computing the complexity increases.

We look at the graph $G_{s_a,w}$, which has a simple structure. Note that the condition $\ell(t_{v(j),i} s_a) - \ell(s_a) =1$ in \eqref{eq-edges-kl-graph} is only satisfied by $s_i\neq s_a$ and the transpositions $t_{a-1,a+1},t_{a,a+2}$. 
It follows that 
    \begin{equation*}
    E(G_{s_a,w})\subset \{(a-1 \to a+1), (a \to a+2) \}\cup\{ (i \to i+1) \mid i\neq a\}.
    \end{equation*}
Moreover, for $i\notin\{a-1,a+1\}$ note that $ s_is_a\le w$ if and only if $s_i\le w$.
By \eqref{eq-edges-kl-graph} it follows that for $i\notin\{a-1,a,a+1\}$ we have  $ (i \to i+1) \in E(G_{s_a,w})$ if and only if $s_i\le w$. While the following result is not formulated in terms of graphs, its proof relies on the above description of~$G_{s_a,w}$.

\begin{thm}\label{vistransp}
The usual torus action on the KL~variety $\KL_{s_a,w}$ is of complexity-$k$
if and only if one of the following holds
\begin{enumerate}
    \item $|\{s_i \mid s_i \le w\}|=\ell(w) - 1-k$ and either $s_{a-1}s_a s_{a-1} \leq w$ or $s_{a+1} s_a s_{a+1} \leq w$.  
    \item $|\{s_i \mid s_i \le w\}|=\ell(w)-k$, $s_{a-1}s_a s_{a-1} \not\leq w$, and $s_{a+1} s_a s_{a+1} \not\leq w$.
\end{enumerate}
\end{thm}

\begin{proof}
Consider the following cases:
\begin{enumerate} 
    \item 
     Suppose that exactly one of $\{ (a-1 \to a), (a-1 \to a+1) \} \subset E(G_{v,w})$ or $\{(a \to a+2), (a+1 \to a+2) \} \subset E(G_{v,w})$ holds. Then, by~\Cref{lem-weights}, $s_{a-1} s_a \leq w$ and $t_{a-1,a+1} s_a = s_a s_{a-1} \leq w$ or $s_{a+1} s_a \leq w$ and $t_{a,a+2} s_a = s_a s_{a+1} \leq w$. 
     By the subword property, this holds if and only if $s_a s_{a-1} s_a=s_{a-1}s_a s_{a-1} \leq w$ or $s_a s_{a+1} s_a =s_{a+1} s_a s_{a+1} \leq w$.
     Since in both cases $G_{v,w}$ is a forest and exactly one of $s_{a-1},s_{a+1}\le w$ holds,
        $$
        \dim(\sigma_{v,w})=|E(\sigma_{v,w})|=1+|\{s_i \mid s_i\le w,\ i\neq a\}|=|\{s_i \mid s_i\le w\}|
        $$
      We conclude that in this case, $\KL_{s_a,w}$ has complexity-$k$ if and only if $|\{s_i \mid s_i\le w\}|=\ell(w)-1-k$.
     \item
     Suppose that $\{(a-1 \to a) , (a-1 \to a+1) , (a \to a+2) , (a+1 \to a+2) \} \subset E(G_{v,w})$. Arguing as above, this holds if and only if $s_{a-1}s_a s_{a-1} \leq w$ and $s_{a+1} s_a s_{a+1} \leq w$.
     In this case, $G_{v,w}$ has exactly one undirected cycle and the number of connected components is $n+1-|E(G_{v,w})|$, i.e.~$\dim(\sigma_{v,w})=|E(G_{v,w})|-1$.
     Moreover, since $s_{a-1},s_{a+1}\le w$ we have
        $$
        |E(G_{v,w})|=2+|\{s_i \mid s_i\le w,\ i\neq a\}|=1+|\{s_i \mid s_i\le w\}|.
        $$
    We conclude that in this case, $\KL_{s_a,w}$ has complexity-$k$ if and only if $|\{s_i \mid s_i\le w\}|=\ell(w)-1-k$.
      \item Suppose that at most one edge of each set $\{ (a-1 \to a) , (a-1 \to a+1) \}$ and $\{ (a\to a+2), (a+1 \to a+2) \}$ is in $E(G_{v,w})$.
      By the previous two cases, this holds if and only if $s_{a-1}s_a s_{a-1} \not\leq w$ and $s_{a+1} s_a s_{a+1} \not\leq w$.
      Since $G_{v,w}$ is a forest we have that 
        $$\dim(\sigma_{v,w})=|E(G_{v,w})|=|\{s_i\mid s_i\le w\}|-1.$$
    It follows that $\KL_{s_a,w}$ has complexity-$k$ if and only if $|\{s_i \mid s_i\le w\}|=\ell(w)-k$. \qedhere
\end{enumerate}
\end{proof}

\begin{cor}\label{simplereflectionclasstoric}
If the KL~variety $\KL_{s_a,w}$ is toric, then it is an affine space.
\end{cor}
\begin{proof}
By Theorem~\ref{vistransp}, the only case in which we do not obtain a forest for $G_{s_a,w}$ is when $s_{a-1}s_a s_{a-1}\leq w$ and $s_{a+1}s_as_{a+1} \leq w$.
In this case, $w$ admits at most $\ell(w)-2$ distinct simple reflections in a reduced word expression.
However, this contradicts that $\KL_{s_a,w}$ is toric, since we would need $|\{s_i \mid s_i \le w\}|=\ell(w) - 1$. We conclude by Proposition~\ref{smoothtoricKL} that $\KL_{s_a,w}$ is an affine space. 
\end{proof}

The next result, which follows from Theorem~\ref{vistransp}~(1), gives certain family of complexity-1 $T$-varieties which are not always affine spaces. The approach with reduced word expressions simplifies producing examples and describing large classes of examples in a compact way.

\begin{cor}\label{cor: complexity1family}
Let $w\in S_n$.
The KL variety $\KL_{s_{n-1},w}$ is a complexity-1 $T$-variety if and only if either
\begin{enumerate}
    \item $s_{n-2}s_{n-1} s_{n-2}\leq w$ and $|\{s_i \mid s_i \le w\}| = \ell(w) -2$, or
    \item $s_{n-2}s_{n-1} s_{n-2}\not\leq w$ and $|\{s_i \mid s_i \le w\}| = \ell(w) -1$.
\end{enumerate}
Moreover, the weight cone of $\KL_{s_{n-1},w}$ is smooth.
\end{cor}

Note that the weight cone of $\KL_{s_{n-1},w}$ is smooth if and only if $G_{s_{n-1},w}$ is a forest (Theorem~\ref{thm-primitive}).

\begin{ex}\label{interestingcomplexity1}
Let $v=s_5=123465$ and $w=s_1s_3s_2s_4s_3s_5s_4=245613$ in $S_6$. The permutation $w$ is irreducible and admits $\ell(w)-2$ distinct simple reflections in its reduced form where $s_3$ has a repetition. This is Case 1 in Theorem~\ref{vistransp} and thus $\KL_{v,w}$ is complexity-1.
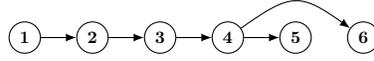
\begin{center}
\begin{tikzpicture}[baseline=1,scale=1.2,every path/.style={>=latex},every node/.style={draw,circle,fill=white,scale=0.6}]

   \node       (v2) at (0.75,0) {\bf{1}};
  \node           (v3) at (1.5,0)  {\bf{2}};

  \node         (v4) at (2.25,0) {\bf{3}};
   \node           (v5) at (3,0)  {\bf{4}};
     \node           (v6) at (3.75,0)  {\bf{5}};
      \node           (v7) at (4.5,0)  {\bf{6}};

 \draw[->] (v4) edge (v5);
  \draw[->] (v3) edge (v4);
   \draw[->] (v5) edge (v6);

    \draw[->] (v2) edge (v3);

  \draw[->] (v5) .. controls (3.70,0.5) and (3.80,0.5) ..  (v7);

\end{tikzpicture}
\end{center}
After imposing the Fulton conditions from $\oDia(w)$ on $Z^{(v)}$, we obtain the second matrix
\[
\begin{bmatrix}
    1       & 0 & 0 & 0  & 0 & 0\\
    z_{21} & 1 & 0 & 0  &0 & 0\\
   z_{31}  & z_{32} & 1 & 0& 0 & 0  \\
  z_{41}       & z_{42} & z_{43} & 1 & 0 & 0 \\
  z_{51} & z_{52} & z_{53}& z_{54}& 0& 1 \\
   z_{61} & z_{62} & z_{63}& z_{64}& 1&0 
\end{bmatrix}
\qquad
\begin{bmatrix}
    1       & 0 & 0 & 0  & 0 & 0\\
    z_{21} & 1 & 0 & 0  &0 & 0\\
  0   & z_{32} & 1 & 0& 0 & 0  \\
  0       & z_{42} & z_{43} & 1 & 0 & 0 \\
  0 & 0 & z_{53}& z_{54}& 0& 1 \\
   0 & 0 & 0 & z_{64}& 1&0 
\end{bmatrix}
\] 
with $z_{32} z_{43} z_{54} - z_{53} z_{32} - z_{54} z_{42} = 0$.
\end{ex}

In Proposition~\ref{irreducibleperm}, we observed that the dimension of weight cone of $\KL_{\id,w}$, i.e.\ the dimension of $\sigma_{\id, w}$ is maximal possible, if $w$ is irreducible. It turns out that this does not hold for $\sigma_{s_a,w}$ but it does for the dimension of the weight cone of $v \Omega_{\id}^\circ \cap X_w$ for all $v \leq w$.

\begin{prop}\label{thm: complexityfix}
Let $w \in S_n$ be an irreducible permutation. Then $\Gamma_w(v)$ is connected for all $v\leq w$. In particular, $v\Omega_{\id}^{\circ} \cap X_w$ is a $T$-variety of complexity $\ell(w)-n+1$ with respect to the $(n-1)$-dimensional usual torus action, i.e.\ as large as possible. 
\end{prop}

\begin{proof}
The complexity does not change for the $T$-invariant open neighborhoods in $X_w$, i.e.\ for $v\Omega_{\id}^{\circ} \cap X_w$ with $v \leq w$. Since $w$ is irreducible, by Proposition~\ref{irreducibleperm}, $\id \Omega_{\id}^{\circ} \cap X_w$ is a complexity-$d$ $T$-variety with connected $\Gamma_w(\id)$, where $d = \ell(w) - n +1$. Thus for any $v\leq w$, $v\Omega_{\id}^{\circ} \cap X_w$ is also a complexity-$d$ $T$-variety with respect to the $(n-1)$-dimensional usual torus action.
\end{proof}

\begin{ex}
For $v = s_3$ and $w=s_1 s_3 s_2 s_3 s_5 s_4=243615$, one obtains that while $\Gamma_w(v)$ is connected, $G_{v,w}$ has 2 connected components.
\end{ex}

\subsection{Observations for the remaining cases}
Let us examine some easy cases first:
\begin{enumerate}
    \item  The KL~variety $\KL_{w,w}$ is the point $wB$ and the graph $G_{v,w}$ is empty so it is a toric variety.
\item If $v=w_0$, then we have only one possible permutation $v=w$ and hence we are in the preceding case.
\item If $w=w_0$, then we obtain that $\KL_{v,w_0} \cong \CC^{|\oDia(v)|}$ and $E(\widetilde{G_{v,w}})$ corresponds to all non-inversions of $v$.
\item Let $v, w \in S_n$ with $v \leq w$ and $\ell(w) - \ell(v) =1$. Then by~\cite[Prop.~1.3.5]{brion}, $\KL_{v,w} \cong \CC^1$. From \eqref{eq-edges-kl-graph} it follows that $\dim(\sigma_{v,w})=|E(G_{v,w})|=1$ and thus $\KL_{v,w}$ is toric.
\end{enumerate}

\begin{rem}\label{rem: concludingremark}
One may think more generally about investigating the complexity of the torus action on a KL variety in the following way. Fix $w\in S_n$ and consider a chain of permutations $v_0=\id < v_1 < \ldots < v_k =w$ where $\ell(v_i)=i$ and $v_{i+1} = s_{j_i} v_{i}$ for some $s_{j_i}$. 
That means that the pair $(j_i,j_i+1)$ is a non-inversion of $v_i$ and an inversion of $v_{i+1}$. 
Suppose that we aim to determine how the complexity of the torus action on $\KL_{v_i,w}$ changes as we go up on the chain.
In the following example, we will see that the change seems difficult to describe.

Let $w=3412$ and consider the chain $v_0=\id < v_1= s_2 < v_2=  s_3 s_2 <  v_3=s_1 s_3 s_2 < s_2 s_1 s_3 s_2 = v_4=w$. 
\begin{itemize}
    \item{$\KL_{v_0,w}$} is the hypersurface defined by~$z_{21} z_{32} z_{43} - z_{42}z_{21} - z_{43} z_{31}$. It is a complexity-1 $T$-variety by Theorem~\ref{visid}.
    \item By Theorem~\ref{vistransp}, {$\KL_{v_1,w}$} is a toric variety and it is given by $z_{21}z_{43} + z_{31} z_{42}$. 
    Note that $(2 \to 3) \in E(G_{v_0,w})\setminus E(G_{v_1,w})$ because $t_{2,3}$ is the only inversion of $s_2$.
    \item{$\KL_{v_2,w}\cong \CC^2 $} is toric. The edge $(3 \to 2) \in E_{w}(v_1)$ becomes an edge of $G_{v_2,w}$ and $(2 \to 4) \in E(G_{v_1,w})$ becomes an edge of $\Gamma_w(v_2)$. The reason for that is  $z_{43}$ from $Z^{(v_1)}$ vanishes in $Z^{(v_2)}$ and $z_{33}$ is a coordinate of $Z^{(v_2)}$ respectively. Moreover, after imposing Fulton's conditions from $\oDia(w)$, we obtain that $z_{31}$ is an unexpected~0, corresponding to the edge $(1 \to 3)$.
    \item{$\KL_{v_3,w}\cong \CC^1$} is toric. The edge $(1 \to 2) \notin G_{v_3,w}$ since $z_{21}$ is not a coordinate of $Z^{(v_3)}$. In particular $z_{33}$ is an unexpected zero corresponding to the edge $(1 \to 3)$. 
\end{itemize}

On the other hand, we can say more about the complexity of each $v_i \Omega_{\id}^\circ\cap X_w$.
By Proposition~\ref{thm: complexityfix}, $\Gamma_w(v_i)$ is connected, thus $v_i \Omega_{\id}^\circ\cap X_w$ is a complexity-1 $T$-variety for all $i \in \{0,1,2,3,4\}$. 

\begin{center}
    
\begin{tikzpicture}[baseline=1,scale=1.15,every path/.style={>=latex},every node/.style={draw,circle,fill=white,scale=0.6}]
   \node       (v1) at (0.75,0) {\bf{1}};
  \node           (v2) at (1.5,0)  {\bf{2}};
 \node         (v3) at (2.25,0) {\bf{3}};
   \node           (v4) at (3,0)  {\bf{4}};
      \node  [draw=none, fill=none, scale=1.5]         (v5) at (1.85,-0.7)  {$G_{v_0,w}$};
  \draw[->] (v1) edge (v2);
   \draw[->] (v2) edge (v3);
   \draw[->] (v3) edge (v4);
\end{tikzpicture}
\qquad
\begin{tikzpicture}[baseline=1,scale=1.15,every path/.style={>=latex},every node/.style={draw,circle,fill=white,scale=0.6}]
   \node       (v1) at (0.75,0) {\bf{1}};
  \node           (v2) at (1.5,0)  {\bf{2}};
  \node         (v3) at (2.25,0) {\bf{3}};
   \node           (v4) at (3,0)  {\bf{4}};

\node  [draw=none, fill=none, scale=1.5]         (v5) at (1.85,-0.7)  {$G_{v_1,w}$};

  \draw[->] (v1) edge (v2);
   \draw[->] (v3) edge (v4);
     \draw[->]  (v1).. controls (1.45,0.5) and (1.55,0.5) .. (v3);
  \draw[->] (v2) .. controls (2.2,-0.5) and (2.3,-0.5) ..  (v4);
    \draw[<-,red,dashed] (v2) edge (v3);

\end{tikzpicture}
\qquad
\begin{tikzpicture}[baseline=1,scale=1.15,every path/.style={>=latex},every node/.style={draw,circle,fill=white,scale=0.6}]

   \node       (v1) at (0.75,0) {\bf{1}};
  \node           (v2) at (1.5,0)  {\bf{2}};

  \node         (v3) at (2.25,0) {\bf{3}};
   \node           (v4) at (3,0)  {\bf{4}};

\node  [draw=none, fill=none, scale=1.5]         (v5) at (1.85,-0.7)  {$G_{v_2,w}$};

  \draw[->] (v1) edge (v2);
   \draw[->] (v2) edge (v3);
   
  \draw[<-,dashed,red] (v2) .. controls (2.2,-0.5) and (2.3,-0.5) ..  (v4);
   \draw[<-,dashed,red] (v3) edge (v4);
  \draw[->, cyan]  (v1).. controls (1.45,0.5) and (1.55,0.5) .. (v3);

\end{tikzpicture}
\qquad
\begin{tikzpicture}[baseline=1,scale=1.15,every path/.style={>=latex},every node/.style={draw,circle,fill=white,scale=0.6}]
   \node       (v1) at (0.75,0) {\bf{1}};
  \node           (v2) at (1.5,0)  {\bf{2}};
  \node         (v3) at (2.25,0) {\bf{3}};
   \node           (v4) at (3,0)  {\bf{4}};
   \node  [draw=none, fill=none, scale=1.5]         (v5) at (1.85,-0.7)  {$G_{v_3,w}$};
  \draw[->] (v2) edge (v3);
  \draw[<-,dashed,red] (v1) edge (v2);
  \draw[<-,dashed,red] (v1) .. controls (2,-0.5) and (2.1,-0.5) .. (v4);
 \draw[<-,dashed,red] (v3) edge (v4);
   \draw[->, cyan]  (v1).. controls (1.45,0.5) and (1.55,0.5) .. (v3);
  
\end{tikzpicture}


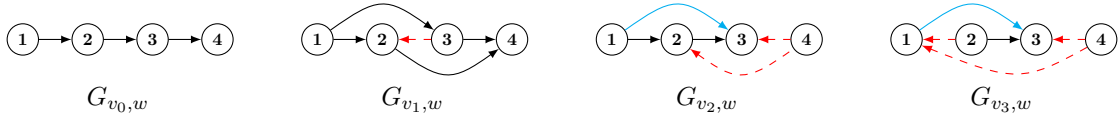
\captionof{figure}{$E_w(v) \backslash E(G_{v,w})$ is represented by red and the edges corresponding to unexpected 0s by blue.}\label{fig-graphs}
\end{center}
\end{rem}

Finally we introduce families of complexity-1 $T$-varieties in terms of reduced word expressions.
\begin{cor}\label{cor_intersection_compl}
Suppose that $w \in S_n$ is irreducible. Assume that a reduced word expression of $w$ contains
\begin{enumerate}
    \item $s_{i+1}s_is_{i+2}s_{i+1}$ as a factor and has no other repetitions. Then for any $v \leq w$, $v\Omega_{\id}^{\circ} \cap X_w$ is a complexity-1 $T$-variety of dimension $n$.
    \item $s_{i}s_{i+1}s_{i}$ as a
factor and has no other repetitions. Then for any $v \leq w$, $v\Omega_{\id}^{\circ} \cap X_w$ is a smooth complexity-1 $T$-variety of dimension $n$.
\end{enumerate}
\end{cor}
\begin{proof}
In (1), $X_w$ is singular and of complexity-1 by~\cite[Thm.~1.3]{leemasudapark}. Thus its open invariant subset $v\Omega_{\id}^{\circ} \cap X_w$ also has complexity-1. Since its complexity is~1, by Proposition~\ref{thm: complexityfix} we have $l(w) = n$, hence it is $n$-dimensional. The argument for (2) is analogous, based on \cite[Thm.~1.2]{leemasudapark}. In this case $X_w$ is smooth, hence its open subset also.
\end{proof}

\begin{ex}
Note that in Corollary~\ref{cor_intersection_compl} (1) the non-smoothness of $X_w$ does not imply that $v\Omega_{\id}^{\circ} \cap X_w$ has a singular point. For example, for $v=w=3412 = s_2 s_1 s_3 s_2$, we obtain a smooth variety $v\Omega_{\id}^{\circ} \cap X_w \simeq \mathbb{C}^4$.
\end{ex}

Although we may define a family of $T$-varieties of fixed dimensional torus action on $v\Omega_{\id}^{\circ} \cap X_w$, the dimension of weight cone of $\KL_{v,w}$ decreases by one or stays the same with this approach, see Remark~\ref{rem: concludingremark}. 
However, our results are sufficient to generate series of KL varieties of small complexity, which is an object of further investigation. Moreover, it is also promising to investigate $\oDia(v)$ by changing its shape under certain restrictions to obtain another $\oDia(v')$ while preserving the complexity of the $T$-varieties. This is an ongoing work of the authors.

%%%%%%%%%%%%%%%%%%%%%%%%%%%%%%%%%%%%%%%%%%%%%

\section*{Acknowledgments}
The authors are grateful for the support, hospitality, and the research environment of the Fields Institute for Research in Mathematics in Toronto, Canada. We are thankful to Klaus Altmann whose question gave rise to this paper, which is about browsing $T$-varieties in ``real life", instead of constructing them via p-divisors.
We would like to thank the anonymous referees for their careful reading of the work and their valuable suggestions. Finally, we are grateful to Alexander Yong for helpful exchanges.

MDB was partially supported by the Polish National Science Center project 2017/26/D/ST1/00755.
LE was partially supported by NSF Grant DMS 1855598 and NSF CAREER Grant DMS 2142656. IP is grateful for the hospitality of University of Warsaw and Washington University in St.\ Louis during the research visits in the process of this paper.

\bibliographystyle{alpha}
\bibliography{references}

\end{document}